\documentclass[11pt,oneside]{article}

\usepackage{amsthm, amsmath, amssymb, amsfonts,epsfig}
\usepackage{mathtools}
\usepackage{esint}
\usepackage{epstopdf}

\usepackage{graphicx}
\usepackage[font=sl,labelfont=bf]{caption}
\usepackage{subcaption}

\usepackage{enumerate}

\usepackage[colorlinks=true, pdfstartview=FitV, linkcolor=blue,
            citecolor=blue, urlcolor=blue]{hyperref}
\usepackage[usenames]{color}
\definecolor{Red}{rgb}{0.7,0,0.1}
\definecolor{Green}{rgb}{0,0.7,0}

\usepackage[numbers]{natbib}

\usepackage{url}
\makeatletter\def\url@leostyle{%
 \@ifundefined{selectfont}{\def\UrlFont{\sf}}{\def\UrlFont{\scriptsize\ttfamily}}} \makeatother

\usepackage{accents, comment}

\usepackage{bbm}

\usepackage{mathrsfs}

\newtheorem{theorem}{Theorem}
\newtheorem{proposition}[theorem]{Proposition}
\newtheorem{lemma}[theorem]{Lemma}
\newtheorem{corollary}[theorem]{Corollary}

\theoremstyle{definition}
\newtheorem{definition}[theorem]{Definition}

\theoremstyle{remark}
\newtheorem{remark}[theorem]{Remark}

\numberwithin{equation}{section}
\numberwithin{theorem}{section}

\def\cC{\mathcal{C}}
\def\cD{\mathcal{D}}

\def\cH{\mathcal{H}}

\def\cM{\mathcal{M}}

\def\cP{\mathcal{P}}

\def\cZ{\mathcal{Z}}

\def\bC{\mathbb{C}}

\def\bR{\mathbb{R}}
\def\bS{\mathbb{S}}

\def\bZ{\mathbb{Z}}

\newcommand{\1}{\mathbbm{1}}                     

\title{Local-in-time Solvability and Space Analyticity for \\ the Navier-Stokes Equations with BMO-type Initial Data}
\author{Liaosha Xu\\
}
\date{}

\begin{document}

\maketitle

\begin{abstract}

\bigskip
It is proved that there exists a local-in-time solution $u\in C([0,T),bmo(\bR^d)^d)$ of the Navier-Stokes equations such that every $u(t)$ has an analytic extension on a complex domain whose size only depends on $t$ (and increases with $t$) and the external force $f$, assuming only that the initial velocity $u_0$ is a local $BMO$ function. Our method for proving is a combination and refinement of the work by \citet{Grujic1998}, \citet{Guberovic2010} and \citet{Kozono2003}. One challenging step is the estimation of the heat and Stokes semigroups from $BMO$-type spaces to $L^\infty$; a result itself of independent interest. We also apply the idea to the analyticity of vorticity with the assistance of Calder\'on-Zygmund theory.

\bigskip

\end{abstract}

\section{Introduction}

We consider the space analyticity of solutions to the Navier-Stokes system in $\bR^d$ ($d\ge 2$)
\begin{align}
&\partial_t u-\Delta u+ u\cdot\nabla u+\nabla p=f, &&\textrm{ in }\bR^d\times (0,T) \label{eq:NSE1}
\\
&\textrm{div}\ u=0,                                &&\textrm{ in }\bR^d\times (0,T) \label{eq:NSE2}
\\
&u(\cdot,0)=u_0(\cdot),                            &&\textrm{ in }\bR^d\times \{t=0\} \label{eq:NSE3}
\end{align}
where the force $f(\cdot,t)$ is real-analytic in space with an uniform analyticity radius $\delta_f$ for all $t\in\bR^+$, which admits some analytic extension $f+ig$, while $u_0$ is the given initial velocity vector field. We prove that there exists a solution $u\in C([0,T^*),bmo(\bR^d)^d)$ for the system \eqref{eq:NSE1}-\eqref{eq:NSE3} evolving from the initial value $u_0\in bmo(\bR^d)^d$, which admits space analyticity in a domain $\cD_t$ of $\bC^d$ for every $t$ such that both real and imaginary parts have bounded local mean oscillations, where the time bound $T^*$ is characterized only by $\|u_0\|_{bmo}$. With certain modification of the proof, the result also holds if $bmo$ space is replaced by $B^0_{\infty,\infty}$ space.

The study of analyticity radius of solutions to the NSE dates back to the seminal work by \citet{Foias1989} who applied Fourier techniques and Gevrey spaces in $L^2$. Related results for $L^p$ spaces can be found in \citet{Lemarie-Rieusset2004}. Similar types of approaches to other equations and function spaces had been developed by \citet{Levermore1997}, \citet{Paicu2011}, \citet{Biswas2007}, \citet{Biswas2014}, \citet{Bae2012} and \citet{Ignatova2012}.
A different method for analyticity of the NSE with $L^p$ ($3<p<\infty$) initial value was first presented by \citet{Grujic1998} using the technique of complexified extension. The idea was adopted to the NSE with $L^\infty$ initial value by \citet{Guberovic2010}. This method was also applied to non-linear heat equations on bounded domains by \citet{Grujic1999} for the analyticity at interior points. Later, the technique was refined by \citet{Bradshaw2015a} for local analyticity of the NSE with locally analytic forcing term.

Among existing literatures of about NSE, very few had addressed the analyticity or even the strong solvability of the Navier-Stokes equations ($d\ge 3$) with non-decaying or oscillatory type initial values. \citet{Giga1999} proved strong solutions exist for $u_0\in BUC$ (bounded and uniformly continuous functions). \citet{Guberovic2010} derived the space analyticity for $u_0\in L^\infty$, showing the time-local existence of mild solution for $L^\infty$ initial value.
\citet{Sawada2003} proved the time-local existence of solutions in Besov spaces with non-positive differential orders using various H\"older-type estimates in the Besov spaces. Later, \citet{Kozono2003} showed the time-local existence in Besov spaces and in ``time-logarithmically-weighted $L^\infty$ spaces'' for $u_0\in B^0_{\infty,\infty}$ using a different converging algorithm. Until now, no result has been established for existence in $C([0,T],X)$ where $X$ is either a local or global  $BMO$-type of space.
This paper is the first attempt to deal with this challenge. We develop a new iteration scheme and converging argument from the approaches as in \citet{Grujic1998}, \citet{Sawada2003} and \citet{Kozono2003}, and consider the strong solvability and spatial analyticity of \eqref{eq:NSE1}-\eqref{eq:NSE3} with $u_0\in bmo(\bR^d)^d$.
Having derived some results on the boundedness of the Stokes semigroups from $BMO$ (or Besov) spaces to $L^\infty$, with the chain of embeddings which connects $L^p$, $BMO$ and Besov spaces, we proved that the complexified solution of \eqref{eq:NSE1}-\eqref{eq:NSE3} exists locally in time with almost $t^{\frac{1}{2}}$ analyticity radius for each $u(t)$, and the real solution, i.e. restriction on the real axis, is classical in the interior of the parabolic cylinder $\bR^d\times(0,T^*)$.

A natural follow-up question is whether we can extend the method to the usual BMO spaces (homogeneous type) and study the global oscillations of the solution to \eqref{eq:NSE1}-\eqref{eq:NSE3}. A positive answer is given in a forthcoming paper. We also attempt to study the local analyticity and the bound of oscillations of the NSE on bounded domains, especially the ones with regular boundaries, i.e. straight and circular boundary lines.
One of the motivations for studying the analyticity estimate and in-time bound of $BMO$-norms is their connection with the sparseness of the region of intense oscillations which can possibly provide a geometric type criterion weaker than the one presented in \citet{Grujic2013} which requires the sparseness of the level sets of velocity or vorticity truncated by the $L^\infty$-norm.
Since the measures of the level sets of oscillation mean are naturally controlled by the $BMO$-norms (e.g. John-Nirenberg inequality), replacing the notion of intense velocity (or vorticity) by that of intense oscillation may potentially reduce the restriction on the geometry of fluid activity near the possible blow-up time. The fundamental step of this idea requires the in-time continuity of $BMO$ norms of velocity and vorticity.

The paper is organized as follows. In Section~\ref{sec:MainResult}, we state the main theorems with the basic setup for complexified solutions and analyticity for both velocity and vorticity of the Navier-Stokes equations, followed by their proofs. In the appendix, we exhibit and develop some auxiliary results about the heat and Stokes semigroups in Besov and BMO-type spaces which we quote in Section~\ref{sec:MainResult}.

\section{Main Results}\label{sec:MainResult}

This section is devoted to the statements and proofs of the main results. Although being formulated in analogous ways, the results for velocity and vorticity are considered in two separated theorems since their proofs require different technics.  For convenience and clarity, we present several technical lemmas in preparation for the main theorems while their proofs are deferred to the appendix. We omit the definitions and some basic properties of the oscillatory function spaces and singular integral operators introduced in the proofs; the reader may refer to \citet{Stein1993}, \citet{Bahouri2011} and \citet{Triebel2010} for details.

The main result for velocity is as follows:

\begin{theorem}\label{th:MainThm}
Assume $u_0\in bmo(\bR^d)$ and $f(\cdot,t)$ is divergence-free and real-analytic in the space variable with the analyticity radius at least $\delta_f$ for all $t\in[0,\infty)$, and the analytic extension $f+ig$ satisfies
\begin{align*}
\Gamma(t):= \sup_{s<t}\sup_{|y|<\delta_f} \left(\|f(\cdot,y,s)\|_{bmo}+\|g(\cdot,y,s)\|_{bmo}\right)<\infty\ .
\end{align*}
Fix a $t_0>0$ and for any $1<M<2$ define
\begin{align}
T_*=\min\left\{\frac{1}{C(M)\|u_0\|_{bmo}^{2}\Phi_1(\|u_0\|_{bmo})}\ ,\ \ \frac{\|u_0\|_{bmo}\Phi_1(\|u_0\|_{bmo})} {C(M)\ \Phi_1(\Gamma(t_0))}\right\}
\end{align}
where $C(M)$ is a constant only depending on $M$ and $\Phi_1(r)$ is some function with logarithmic growth as $r\to\infty$, both of which are independent of $u_0$ and $f$. Then there exists a solution
\begin{align*}
u\in C([0,T_*),bmo(\bR^d)^d)
\end{align*}
of the NSE \eqref{eq:NSE1}-\eqref{eq:NSE3} such that for every $t\in (0,T_*)$, $u$ is a restriction of an analytic function $u(x,y,t)+iv(x,y,t)$ in the region
\begin{align*}
\cD_t=: \left\{(x,y)\in\bC^d\ \big|\ |y|\le \min\{ct^{1/2}\Phi_2(t)/c(M),\delta_f\}\right\}\ .
\end{align*}
where $c(M)$ is a constant only depending on $M$ and $\Phi_2(t)$ is an explicitly defined function with logarithmic growth as $t\to 0^+$ which will be given in the proof. Moreover,
\begin{align}
&\underset{t\in(0,T)}{\sup}\ \underset{y\in\cD_t}{\sup}\|u(\cdot,y,t)\|_{bmo} + \underset{t\in(0,T)}{\sup}\ \underset{y\in\cD_t}{\sup} \|v(\cdot,y,t)\|_{bmo}\le M\cdot\|u_0\|_{bmo}\ ,
\\
&\underset{t\in(0,T)}{\sup}\ \underset{y\in\cD_t}{\sup}\ \phi_1(t)\|u(\cdot,y,t)\|_{L^\infty} + \underset{t\in(0,T)}{\sup}\ \underset{y\in\cD_t}{\sup}\ \phi_1(t)\|v(\cdot,y,t)\|_{L^\infty}\le M\cdot\|u_0\|_{bmo}\
\end{align}
where $\phi_1(t)=[\ln(e+1/t)]^{-1}$.
\end{theorem}

An analogous result for vorticity is as follows: We consider the vorticity-velocity formulation of the 3D Navier-Stokes equations:
\begin{align}\label{eq:VelVorForm}
\partial_t \omega-\Delta \omega =\omega\nabla u - u\nabla\omega,\qquad \omega(0,x)=\omega_0
\end{align}
and we have
\begin{theorem}\label{th:MainThmVor}
Assume $\omega_0\in bmo(\bR^3)\cap L^p(\bR^3)$ where $1\le p<3$. For any $1<M<2$ let
\begin{align}
T_\omega= C(M)^{-1}\left(\|\omega_0\|_{bmo}+\|\omega_0\|_{L^p}\right)^{-2} \Phi_1(\|\omega_0\|_{bmo}+\|\omega_0\|_{L^p})
\end{align}
where $C(M)$ is a constant depending on $M$ and $\Phi_1(r)$ is some function with logarithmic growth as $r\to\infty$, both of which are independent of $\omega_0$. Then there exists a solution
\begin{align*}
\omega\in C([0,T_\omega),bmo(\bR^3)^3)
\end{align*}
for the NSE \eqref{eq:NSE1}-\eqref{eq:NSE3} such that for every $t\in (0,T_\omega)$, $u$ is a restriction of an analytic function $\omega(x,y,t)+i\zeta(x,y,t)$ in the region
\begin{align}\label{eq:ComplexDomVor}
\cD_t=: \left\{(x,y)\in\bC^3\ \big|\ |y|\le ct^{1/2}\Phi_2(t)/c(M) \right\}\
\end{align}
for some constant $c(M)$, where $\Phi_2(t)$ is an explicitly defined function with logarithmic growth as $t\to 0^+$ as given in Theorem~\ref{th:MainThm}. Moreover,
\begin{align}
&\underset{t\in(0,T)}{\sup}\ \underset{y\in\cD_t}{\sup}\|\omega(\cdot,y,t)\|_{bmo} + \underset{t\in(0,T)}{\sup}\ \underset{y\in\cD_t}{\sup} \|\zeta(\cdot,y,t)\|_{bmo}\le M\cdot \|\omega_0\|_{bmo}\ , \label{eq:VorBddbmo}
\\
&\underset{t\in(0,T)}{\sup}\ \underset{y\in\cD_t}{\sup}\ \phi_1(t)\|\omega(\cdot,y,t)\|_{L^\infty} + \underset{t\in(0,T)}{\sup}\ \underset{y\in\cD_t}{\sup}\ \phi_1(t)\|\zeta(\cdot,y,t)\|_{L^\infty}\le M\cdot \|\omega_0\|_{bmo}\  \label{eq:VorBddLinfty}
\end{align}
where $\phi_1(t)=[\ln(e+1/t)]^{-1}$.
\end{theorem}

An improvement of the above result is the following.
\begin{theorem}\label{th:MainThmVor2}
Assume $\omega_0\in bmo(\bR^3)\cap L^p(\bR^3)$ where $1\le p<3$. For any $1<M<2$ let
\begin{align}
T= \bar C(M)^{-1} \left(\|\omega_0\|_{bmo}+\|\omega_0\|_{L^p}\right)^{-1} \bar\Phi_1(\|\omega_0\|_{bmo}+\|\omega_0\|_{L^p})
\end{align}
where $\bar C(M)$ is a constant depending on $M$ and $\bar\Phi_1(r)$ is some function with logarithmic growth as $r\to\infty$, both of which are independent of $\omega_0$. Then there exists a solution
\begin{align*}
\omega\in C([0,T),bmo(\bR^3)^3)
\end{align*}
for the NSE \eqref{eq:NSE1}-\eqref{eq:NSE3} such that for every $t\in (0,T)$, $u$ is a restriction of an analytic function $\omega(x,y,t)+i\zeta(x,y,t)$ in a region $\tilde\cD_t$ approximated by $\cD_t$ in \eqref{eq:ComplexDomVor} for some constant $\bar c(M)$ and some function $\bar\Phi_2(t)$ with logarithmic growth as $t\to 0^+$. And results analogous to \eqref{eq:VorBddbmo} and \eqref{eq:VorBddLinfty} hold with some logarithmic factor $\bar\phi_1$.
\end{theorem}

\begin{remark}
With some modification of the proof, one can show there exists some $T_*$ such that the equations \eqref{eq:NSE1}-\eqref{eq:NSE3} has an analytic solution
\begin{align*}
u\in C([0,T_*),B^0_{\infty,\infty}(\bR^d)^d)
\\
\left(\textrm{resp. }\omega\in C([0,T_\omega),B^0_{\infty,\infty}(\bR^3)^3)\ \right)
\end{align*}
with the initial value $u_0\in B^0_{\infty,\infty}(\bR^d)$ (resp. $\omega_0\in B^0_{\infty,\infty}(\bR^3)\cap L^p(\bR^3)$) (all the other assumptions remain the same). The work required is similar and easier, we will omit the proof.
\end{remark}

Before proving the main theorems, we present two lemmas which are key results for constructing approximation sequences in $BMO$-spaces with proofs given in the appendix.

\begin{lemma}\label{le:semigplocalbmo}
Consider the equation $\partial_t u-\Delta u =0 $ in $\bR^d\times [0,\infty)$ with $u(0) =u_0\in bmo(\bR^d)$. There is a solution $u(t)$ and constant $C$ satisfying the estimate
\begin{align}
\underset{t>0}{\sup}\left(\|u(t)\|_{bmo} +t^{\frac{1}{2}}\|\nabla u(t)\|_{L^\infty} + t\|\nabla^2 u(t)\|_{L^\infty} + t\|\partial_t u\|_{L^\infty}\right)\le C\|u_0\|_{bmo}\ .
\end{align}
\end{lemma}

\begin{lemma}\label{le:RealAnalytic}
Let $\partial_t u-\Delta u= \nabla\cdot R(f_t)$ with $u(0)\in bmo(\bR^d)^d$, where $R$ is a Calder\'on-Zygmund operator (see \citet{Stein1993} for the definition of C.Z.O.) and $f_t$ is an analytic function for every $t\in(0,T)$ and has bounded-in-time $BMO$-norms, weighted in time; more precisely, there is a continuous function $a_0(t)\gtrsim t^{1-\epsilon}$ such that
\begin{align}\label{eq:LeRANonHomAsmp}
a_0(t)f_t\in C^\infty([\delta,T),C^\infty(\bR^d)^d)\cap L^\infty([0,T),BMO(\bR^d)^d)
\end{align}
where $\delta>0$ can be arbitrarily small. Then the function
\begin{align}\label{eq:DuhPrin}
u(t)=e^{t\Delta}u_0+ \int_0^t \nabla e^{(t-s)\Delta} R(f_s)\ ds
\end{align}
solves the equations in $\bR^d\times (0,T)$ and is real analytic for every $t\in(0,T)$. Moreover, for any $\delta>0$
\begin{align*}
a_1(t)u\in C^\infty([\delta,T),C^\infty(\bR^d)^d)\cap L^\infty([0,T),L^\infty(\bR^d)^d)\
\end{align*}
with some weight function $a_1(t)\gtrsim[\ln(e+1/t)]^{-1}$.
\end{lemma}

\medskip

\begin{proof}[Proof of Theorem~\ref{th:MainThm}]
We construct an approximating sequence as follows:
\begin{align*}
&u^{(0)}=0\ ,\quad \pi^{(0)}=0\ ,
\\
&\partial_t u^{(n)} -\Delta u^{(n)} = -\left(u^{(n-1)}\cdot \nabla\right)u^{(n-1)} - \nabla\pi^{(n-1)} +f\ ,
\\
&u^{(n)}(x,0)=u_0(x)\ ,\quad \nabla\cdot u^{(n)}=0\ ,
\\
&\Delta\pi^{(n)}=-\partial_j\partial_k\left(u_j^{(n)} u_k^{(n)}\right)\ .
\end{align*}
By induction and Lemma~\ref{le:RealAnalytic}, we know there are $a^{(n)}(t)$ such that
\begin{align*}
a^{(n)}(t)\cdot u^{(n)}\in C^\infty([\delta,T),C^\infty(\bR^d)^d)\cap L^\infty([0,T),L^\infty(\bR^d)^d)\
\end{align*}
for arbitrarily small $\delta$ and each $u^{(n)}(t)$ is real analytic for every $t$. Let $u^{(n)}(x,y,t)+iv^{(n)}(x,y,t)$ and $\pi^{(n)}(x,y,t)+i\rho^{(n)}(x,y,t)$ be the analytic extensions of $u^{(n)}$ and $\pi^{(n)}$ respectively. Inductively we have analytic extensions for all approximate solutions and the real and imaginary parts satisfy
\begin{align}
\partial_t u^{(n)}-\Delta u^{(n)} &= -\left(u^{(n-1)}\cdot \nabla\right)u^{(n-1)} + \left(v^{(n-1)}\cdot \nabla\right)v^{(n-1)} - \nabla\pi^{(n-1)} + f\ ,
\\
\partial_t v^{(n)}-\Delta v^{(n)} &= -\left(u^{(n-1)}\cdot \nabla\right)v^{(n-1)} - \left(v^{(n-1)}\cdot \nabla\right)u^{(n-1)} - \nabla\rho^{(n-1)} + g\ ,
\end{align}
where
\begin{align*}
\Delta \pi^{(n)} &= -\partial_j\partial_k\left(u^{(n)}_ju^{(n)}_k-v^{(n)}_jv^{(n)}_k\right),\qquad \Delta\rho^{(n)} = -2\partial_j\partial_k\left(u^{(n)}_jv^{(n)}_k\right)\ .
\end{align*}
Now define
\begin{align*}
U_\alpha^{(n)}(x,t)&= u^{(n)}(x,\alpha t,t), && \Pi_\alpha^{(n)}(x,t)= \pi^{(n)}(x,\alpha t,t), & F_\alpha(x,t)=f(x,\alpha t, t),
\\
V_\alpha^{(n)}(x,t)&= v^{(n)}(x,\alpha t,t), && R_\alpha^{(n)}(x,t)= \rho^{(n)}(x,\alpha t,t), & G_\alpha(x,t)=g(x,\alpha t, t),
\end{align*}
then the approximation scheme becomes (for simplicity we drop the subscript $\alpha$)
\begin{align*}
\partial_t U^{(n)}- \Delta U^{(n)} &= -\alpha\cdot \nabla V^{(n)}-\left(U^{(n-1)}\cdot \nabla\right)U^{(n-1)} + \left(V^{(n-1)}\cdot \nabla\right)V^{(n-1)} - \nabla\Pi^{(n-1)} + F\ ,
\\
\partial_t V^{(n)}- \Delta V^{(n)} &= -\alpha\cdot \nabla U^{(n)}-\left(U^{(n-1)}\cdot \nabla\right)V^{(n-1)} - \left(V^{(n-1)}\cdot \nabla\right)U^{(n-1)} - \nabla R^{(n-1)} + G\ ,
\\
\Delta \Pi^{(n)} &= -\partial_j\partial_k\left(U^{(n)}_jU^{(n)}_k-V^{(n)}_jV^{(n)}_k\right),\qquad \Delta R^{(n)} = -2\partial_j\partial_k\left(U^{(n)}_jV^{(n)}_k\right)\ .
\end{align*}
with initial conditions
\begin{align*}
U^{(n)}(x,0)=u_0(x),\qquad V^{(n)}(x,0)=0\qquad\textrm{for all }x\in\bR^d\ ,
\end{align*}
for which we have the following iteration:
\begin{align}
U^{(n)}(x,t)& =e^{t\Delta} u_0 - \int_0^te^{(t-s)\Delta}\left(U^{(n-1)}\cdot\nabla\right)U^{(n-1)} ds + \int_0^te^{(t-s)\Delta}\left(V^{(n-1)}\cdot\nabla\right)V^{(n-1)} ds \notag
\\
&\quad -\int_0^te^{(t-s)\Delta}\nabla\Pi^{(n-1)}ds + \int_0^te^{(t-s)\Delta} F\ ds - \int_0^te^{(t-s)\Delta}\alpha\cdot\nabla V^{(n)}ds\ , \label{eq:IterationU}
\\
V^{(n)}(x,t)& = - \int_0^te^{(t-s)\Delta}\left(U^{(n-1)}\cdot\nabla\right)V^{(n-1)} ds - \int_0^te^{(t-s)\Delta}\left(V^{(n-1)}\cdot\nabla\right)U^{(n-1)} ds \notag
\\
&\quad -\int_0^te^{(t-s)\Delta}\nabla R^{(n-1)}ds + \int_0^te^{(t-s)\Delta} G\ ds - \int_0^te^{(t-s)\Delta}\alpha\cdot\nabla U^{(n)}ds \label{eq:IterationV}
\end{align}
where
\begin{align*}
\Pi^{(n)}(x,t)& =-(\Delta)^{-1}\sum \partial_j\partial_k\left(U^{(n)}_jU^{(n)}_k-V^{(n)}_jV^{(n)}_k\right),
\\
R^{(n)}(x,t) &=-2(\Delta)^{-1}\sum \partial_j\partial_k\left(U^{(n)}_jV^{(n)}_k\right).
\end{align*}

We state that, for some $T$ depending only on $\|u_0\|_{bmo}$, $\|F\|_{bmo}$ and $\|G\|_{bmo}$, the sequences $U^{(n)}$, $V^{(n)}$ constructed as above have a common upper bound in the four types of function spaces: $B^0_{\infty,\infty}$, $bmo$, $\phi(t)L^\infty$ and $\psi(t) \dot B^1_{\infty,1}$. \textit{More precisely, our claim is: There exists $T$ such that for all $n$
\begin{align}
U^{(n)},V^{(n)} &\in C([0,T);bmo(\bR^d)^d)\ , \label{eq:UnVnBMO}
\\
\phi_1(t)U^{(n)},\phi_1(t)V^{(n)} &\in C([0,T);L^\infty(\bR^d)^d)\ , \label{eq:UnVnLinfty}
\\
\phi_2(t)U^{(n)},\phi_2(t)V^{(n)} &\in C([0,T);\dot B^1_{\infty,1}(\bR^d)^d) \label{eq:UnVnBesov}
\end{align}
where $\phi_1(t)$ is given in Theorem~\ref{th:MainThm} and $\phi_2(t)=t^{\frac{1}{2}}$. Moreover
\begin{align*}
&L_n:=\underset{t<T}{\sup}\ \phi_1(t)\|U^{(n)}\|_{L^\infty}+\underset{t<T}{\sup}\ \phi_1(t)\|V^{(n)}\|_{L^\infty}\ ,\quad L_n':=\underset{t<T}{\sup}\ \|U^{(n)}\|_{B^0_{\infty,\infty}}+\underset{t<T}{\sup}\ \|V^{(n)}\|_{B^0_{\infty,\infty}}\ ,
\\
&L_n'':=\underset{t<T}{\sup}\ \|U^{(n)}\|_{bmo}+\underset{t<T}{\sup}\ \|V^{(n)}\|_{bmo}\ ,
\quad L_n''':=\underset{t<T}{\sup}\ \phi_2(t)\|U^{(n)}\|_{\dot B^1_{\infty,1}}+\underset{t<T}{\sup}\ \phi_2(t)\|V^{(n)}\|_{\dot B^1_{\infty,1}}
\end{align*}
are all bounded by a constant only determined by $\|u_0\|_{bmo}$, $\|F\|_{bmo}$ and $\|G\|_{bmo}$.}

\textit{Proof of the claim:} At the initial step of the iteration, i.e.
\begin{align}
U^{(0)}(x,t)& =e^{t\Delta} u_0 -\int_0^te^{(t-s)\Delta} F\ ds - \int_0^te^{(t-s)\Delta}\alpha\cdot\nabla V^{(0)}ds\ , \label{eq:IniIterU}
\\
V^{(0)}(x,t)& = \int_0^te^{(t-s)\Delta} G\ ds - \int_0^te^{(t-s)\Delta}\alpha\cdot\nabla U^{(0)}ds\ , \label{eq:IniIterV}
\end{align}
we have the following chain of estimates which follows by Lemma~\ref{le:SemGpBMO} and Lemma~\ref{le:semigplocalbmo}:
\begin{align}
\|U^{(0)}\|_{bmo}& \lesssim \|e^{t\Delta}u_0\|_{bmo} + \int_0^t\| e^{(t-s)\Delta} F \|_{bmo}ds + \int_0^t\| e^{(t-s)\Delta} \alpha\cdot \nabla V^{(0)} \|_{bmo}ds \notag
\\
& \lesssim \|u_0\|_{bmo} + \int_0^t \|F\|_{bmo}ds + |\alpha|\int_0^t\| \nabla e^{(t-s)\Delta} V^{(0)} \|_{L^\infty}ds \notag
\\
& \lesssim \|u_0\|_{bmo} + t\ \underset{s<t}{\sup}\ \|F\|_{bmo} + |\alpha|\int_0^t(t-s)^{-1/2}\|V^{(0)} \|_{BMO}ds \notag
\\
& \lesssim \|u_0\|_{bmo} + t\ \underset{s<t}{\sup}\ \|F\|_{bmo} + |\alpha|t^{1/2}\underset{s<t}{\sup}\ \|V^{(0)} \|_{bmo}\ . \label{eq:bmoEstU0}
\end{align}
Similarly,
\begin{align}\label{eq:bmoEstV0}
\|V^{(0)}\|_{bmo} \lesssim\|u_0\|_{bmo} + t\ \underset{s<t}{\sup}\ \|G\|_{bmo} + |\alpha|t^{1/2}\underset{s<t}{\sup}\ \|U^{(0)}\|_{bmo}\ .
\end{align}
If we assume $\alpha$ is a vector such that $C|\alpha|t^{1/2}<1/2$ for all $t<T$ with some proper choice of $C$ according to the above estimations, then combining \eqref{eq:bmoEstU0} and \eqref{eq:bmoEstV0} gives
\begin{align}\label{eq:InitialBMO}
\underset{t<T}{\sup}\ \|U^{(0)}\|_{bmo}+\underset{t<T}{\sup}\ \|V^{(0)}\|_{bmo}\lesssim\|u_0\|_{bmo} + T \left(\underset{t<T}{\sup}\|F\|_{bmo}+\underset{t<T}{\sup}\|G\|_{bmo}\right)\ .
\end{align}
Since $bmo\hookrightarrow B^0_{\infty,\infty}$ (see Lemma~\ref{le:BMOEmbedBesov}), the above estimate implies
\begin{align}\label{eq:InitialBesov}
\underset{t<T}{\sup}\ \|U^{(0)}\|_{B^0_{\infty,\infty}}+\underset{t<T}{\sup}\ \|V^{(0)}\|_{B^0_{\infty,\infty}}\lesssim\|u_0\|_{bmo} + T \left(\underset{t<T}{\sup}\|F\|_{bmo}+\underset{t<T}{\sup}\|G\|_{bmo}\right)\ .
\end{align}
The estimations in ``$\phi_1(t)L^\infty$'' are similar: By taking the $L^\infty$-norm of \eqref{eq:IniIterU} and using the fact that $B^0_{\infty,1}\hookrightarrow L^\infty$ (see \citet{Kozono2003}), we obtain
\begin{align*}
\|U^{(0)}\|_{L^\infty}& \lesssim \|e^{t\Delta}u_0\|_{L^\infty} + \int_0^t\| e^{(t-s)\Delta} F \|_{L^\infty}ds + \int_0^t\| e^{(t-s)\Delta} \alpha\cdot \nabla V^{(0)} \|_{L^\infty}ds
\\
& \lesssim \|e^{t\Delta}u_0\|_{B^0_{\infty,1}} + \int_0^t\| e^{(t-s)\Delta}F \|_{B^0_{\infty,1}}ds + |\alpha|\int_0^t\| \nabla e^{(t-s)\Delta} V^{(0)} \|_{{B^0_{\infty,1}}}ds\ .
\end{align*}
Then, by Lemma~\ref{le:HolderEstBesov}
\begin{align*}
\|U^{(0)}\|_{L^\infty}& \lesssim \ln(e+1/t)\|u_0\|_{B^0_{\infty,\infty}} + \int_0^t \ln(e+1/(t-s))\|F\|_{B^0_{\infty,\infty}}ds
\\
&\qquad\qquad + |\alpha|\int_0^t (t-s)^{-\frac{1}{2}}\ln(e+1/(t-s))\| V^{(0)} \|_{{B^0_{\infty,\infty}}}ds
\\
& \lesssim \ln(e+1/t)\|u_0\|_{B^0_{\infty,\infty}} + t\psi_1(t)\ \underset{s<t}{\sup}\ \|F \|_{B^0_{\infty,\infty}} + |\alpha|t^{\frac{1}{2}}\psi_2(t) \| V^{(0)} \|_{B^0_{\infty,\infty}}\ .
\end{align*}
where $\psi_1, \psi_2$ are given explicitly by
\begin{align*}
\psi_1(t)=t^{-1}\int_0^t \ln(e+1/(t-s))ds\ ,\qquad \psi_2(t)=t^{-\frac{1}{2}}\int_0^t (t-s)^{-\frac{1}{2}}\ln(e+1/(t-s))ds\ .
\end{align*}
Since $L^\infty\hookrightarrow bmo\hookrightarrow B^0_{\infty,\infty}$ (see Lemma~\ref{le:BMOEmbedBesov}) we end up with
\begin{align}\label{eq:InfEstU0}
\|U^{(0)}\|_{L^\infty} & \lesssim \ln(e+1/t)\|u_0\|_{bmo} + t\psi_1(t)\ \underset{s<t}{\sup}\ \|F \|_{bmo} + |\alpha|t^{\frac{1}{2}} \psi_2(t) \| V^{(0)} \|_{L^\infty}\ .
\end{align}
It follows from the same reasoning that
\begin{align}\label{eq:InfEstV0}
\|V^{(0)}\|_{L^\infty} \lesssim \ln(e+1/t)\|u_0\|_{bmo} + t\psi_1(t)\ \underset{s<t}{\sup}\ \|G \|_{bmo} + |\alpha|t^{\frac{1}{2}}\psi_2(t)  \| U^{(0)} \|_{L^\infty}\ .
\end{align}
In the rest of the proof we will always write $\phi_1(t)$ for $[\ln(e+1/t)]^{-1}$. If we assume $\alpha$ is such that $C|\alpha|t^{1/2}\psi_2(t) <1/2$ for all $t<T$ with some proper choice of $C$, then \eqref{eq:InfEstU0} and \eqref{eq:InfEstV0} imply that
\begin{align}\label{eq:InitialLinfty}
\underset{t<T}{\sup}\ \phi_1(t)\|U^{(0)}\|_{L^\infty}+\underset{t<T}{\sup}\ \phi_1(t) \|V^{(0)}\|_{L^\infty}\lesssim \|u_0\|_{bmo} + T\psi_3(T) \left(\underset{t<T}{\sup} \|F\|_{bmo}+\underset{t<T}{\sup} \|G\|_{bmo}\right)\ .
\end{align}
where $\psi_3(t)=\phi_1(t)\psi_1(t)$.
For the estimations in $\dot B^1_{\infty,1}$: In virtue of \eqref{eq:HeatHomBesov}, we deduce
\begin{align*}
\|U^{(0)}\|_{\dot B^1_{\infty,1}}& \lesssim \|e^{t\Delta}u_0\|_{\dot B^1_{\infty,1}} + \int_0^t\| e^{(t-s)\Delta} F \|_{\dot B^1_{\infty,1}}ds + \int_0^t\| e^{(t-s)\Delta} \alpha\cdot \nabla V^{(0)} \|_{\dot B^1_{\infty,1}}ds
\\
& \lesssim t^{-1/2}\|u_0\|_{\dot B^0_{\infty,\infty}} + \int_0^t (t-s)^{-\frac{1}{2}}\|F \|_{\dot B^0_{\infty,\infty}}ds + |\alpha|\int_0^t\left(t-s\right)^{-\frac{1}{2}}\| \nabla V^{(0)} \|_{{\dot B^0_{\infty,\infty}}}ds\ .
\end{align*}
Since $\dot B^1_{\infty,1}\hookrightarrow \dot B^0_{\infty,\infty}$ (trivial fact by the definitions) and $bmo\hookrightarrow BMO \hookrightarrow \dot B^0_{\infty,\infty}$ (see Lemma~\ref{le:BMOEmbedBesov})
\begin{align*}
\|U^{(0)}\|_{\dot B^1_{\infty,1}}& \lesssim t^{-1/2}\|u_0\|_{BMO} + \int_0^t (t-s)^{-\frac{1}{2}}\| F \|_{BMO}ds + |\alpha|\int_0^t (t-s)^{-\frac{1}{2}}\|\nabla V^{(0)} \|_{\dot B^0_{\infty,1}}ds
\\
& \lesssim t^{-1/2}\|u_0\|_{bmo} + t^{\frac{1}{2}}\ \underset{s<t}{\sup}\ \|F \|_{bmo} + |\alpha|\ \underset{s<t}{\sup}\ s^{\frac{1}{2}}\|V^{(0)} \|_{\dot B^1_{\infty,1}}\int_0^t (t-s)^{-\frac{1}{2}}s^{-\frac{1}{2}}ds\ .
\end{align*}
Thus, for all $t<T$,
\begin{align}
t^{\frac{1}{2}}\|U^{(0)}\|_{\dot B^1_{\infty,1}}& \lesssim \|u_0\|_{bmo} + t\ \underset{s<t}{\sup}\ \|F \|_{bmo} + |\alpha|t^{\frac{1}{2}}\ \underset{s<t}{\sup}\ s^{\frac{1}{2}}\|V^{(0)} \|_{\dot B^1_{\infty,1}}\ .
\end{align}
Similarly, for all $t<T$,
\begin{align}
t^{\frac{1}{2}}\|V^{(0)}\|_{\dot B^1_{\infty,1}}& \lesssim \|u_0\|_{bmo} + t\ \underset{s<t}{\sup}\ \|F \|_{bmo} + |\alpha|t^{\frac{1}{2}}\ \underset{s<t}{\sup}\ s^{\frac{1}{2}}\|U^{(0)} \|_{\dot B^1_{\infty,1}}\ .
\end{align}
Again, choosing $\alpha$ properly, combination of the above two estimates shows that
\begin{align}\label{eq:InitialBinfty1}
\underset{t<T}{\sup}\ t^{\frac{1}{2}} \|U^{(0)}\|_{\dot B^1_{\infty,1}}+\underset{t<T}{\sup}\ t^{\frac{1}{2}}\|V^{(0)}\|_{\dot B^1_{\infty,1}}\lesssim\|u_0\|_{bmo} + T \left(\underset{t<T}{\sup}\|F\|_{bmo}+\underset{t<T}{\sup}\|G\|_{bmo}\right)\ .
\end{align}
To sum up the estimates \eqref{eq:InitialBMO}, \eqref{eq:InitialBesov}, \eqref{eq:InitialLinfty} and \eqref{eq:InitialBinfty1} we've shown that
\begin{align*}
M_0:=\max\{L_0,L_0',L_0'',L_0'''\}\lesssim \|u_0\|_{bmo} + T\psi_3(T)  \left(\underset{t<T}{\sup}\|F\|_{bmo}+\underset{t<T}{\sup}\|G\|_{bmo}\right)\ .
\end{align*}
To control the rest of $M_n$'s in the iteration scheme, estimations for the nonlinear terms play an essential role. In the following we will demonstrate the idea with the terms such that
\begin{align*}
\int_0^t e^{(t-s)\Delta}(U^{(n)}\cdot \nabla)U^{(n)}ds\ , \quad \int_0^t e^{(t-s)\Delta}(U^{(n)}\cdot \nabla)V^{(n)}ds\ , \quad \int_0^te^{(t-s)\Delta}\nabla\Pi^{(n)}ds\ , \quad ...
\end{align*}
in those four spaces (we won't write down details of the estimation for each term in each function space because some of the ideas and techniques are similar). First we will derive estimation in $L^\infty$; we will see the results for the other spaces essentially follow from it.

Recall that $\nabla\cdot U^{(n)}=0$, so $(U^{(n)}\cdot\nabla) U^{(n)}=\nabla(U^{(n)}\otimes U^{(n)})$. For induction hypothesis, we suppose \eqref{eq:UnVnBMO}-\eqref{eq:UnVnBesov} holds true for $n$. Then, by Lemma~\ref{le:SemGpBMO}
\begin{align*}
\left\|\int_0^t e^{(t-s)\Delta}(U^{(n)}\cdot \nabla)U^{(n)}ds\right\|_{L^\infty}&\lesssim\int_0^t \|e^{(t-s)\Delta}\nabla (U^{(n)}\otimes U^{(n)})\|_{L^\infty}ds
\\
&\lesssim\int_0^t \|\nabla e^{(t-s)\Delta} (U^{(n)}\otimes U^{(n)})\|_{L^\infty}ds
\\
&\lesssim\int_0^t (t-s)^{-1/2}\|U^{(n)}\otimes U^{(n)}\|_{BMO}ds
\\
&\lesssim\int_0^t (t-s)^{-1/2}\|U^{(n)}\otimes U^{(n)}\|_{L^\infty}ds
\\
&\lesssim \left(\underset{s<t}{\sup}\ \phi_1(s)\|U^{(n)}\|_{L^\infty}\right)^2\int_0^t (t-s)^{-1/2}(\phi_1(s))^{-2}ds\ .
\end{align*}
Note that the estimates for $V^{(n)}$ follow in the same way, thus we have
\begin{align}
\phi_1(t)\left\|\int_0^t e^{(t-s)\Delta}(U^{(n)}\cdot \nabla)U^{(n)}ds\right\|_{L^\infty}&\lesssim t^{\frac{1}{2}}\psi_4(t)\left(\underset{s<t}{\sup}\ \phi_1(s)\|U^{(n)}\|_{L^\infty}\right)^2\ ,\label{eq:UnLinftyUnNL}
\\
\phi_1(t)\left\|\int_0^t e^{(t-s)\Delta}(V^{(n)}\cdot \nabla)V^{(n)}ds\right\|_{L^\infty}&\lesssim t^{\frac{1}{2}}\psi_4(t)\left(\underset{s<t}{\sup}\ \phi_1(s)\|V^{(n)}\|_{L^\infty}\right)^2 \label{eq:UnLinftyVnNL}
\end{align}
where
\begin{align*}
\psi_4(t)=t^{-\frac{1}{2}}\phi_1(t)\int_0^t (t-s)^{-1/2}(\phi_1(s))^{-2}ds\ .
\end{align*}
The bound for the pressure term is obtained in a similar way: For induction hypothesis, we suppose \eqref{eq:UnVnBMO}-\eqref{eq:UnVnBesov} hold true for $n$. For convenience we will denote by $P$ the projection operator, i.e.
\begin{align*}
P(f\otimes g):=-(\Delta)^{-1}\sum \partial_j\partial_k(f_j\cdot g_k)\ .
\end{align*}
Since $\|\nabla f\|_\infty\lesssim\|\nabla f\|_{\dot B^0_{\infty,1}}$ if $f\in BMO$ and $\nabla f\in\dot B^0_{\infty,1}$ (see \citet{Kozono2003}), we have
\begin{align*}
\left\|\int_0^te^{(t-s)\Delta}\nabla\Pi^{(n)}ds\right\|_{L^\infty}&\lesssim\int_0^t \left\|\nabla e^{(t-s)\Delta} P\left(U^{(n)}\otimes U^{(n)}-V^{(n)}\otimes V^{(n)}\right)\right\|_{L^\infty}ds
\\
&\lesssim\int_0^t \left\|\nabla e^{(t-s)\Delta} P\left(U^{(n)}\otimes U^{(n)}-V^{(n)}\otimes V^{(n)}\right)\right\|_{\dot B^0_{\infty,1}}ds\ .
\end{align*}
Note that $P$ is a bounded operator from $\dot B^0_{\infty,\infty}$ into itself (see \citet{Lemarie1985} and \citet{Han1993}). Then, by Lemma~\ref{eq:HeatHomBesovLog} and Lemma~\ref{le:BMOEmbedBesov} it follows that
\begin{align*}
\left\|\int_0^te^{(t-s)\Delta}\nabla\Pi^{(n)}ds\right\|_{L^\infty}&\lesssim\int_0^t (t-s)^{-1/2}\left\|P\left(U^{(n)}\otimes U^{(n)}-V^{(n)}\otimes V^{(n)}\right)\right\|_{\dot B^0_{\infty,\infty}}ds
\\
&\lesssim\int_0^t (t-s)^{-1/2}\left\|U^{(n)}\otimes U^{(n)}-V^{(n)}\otimes V^{(n)}\right\|_{\dot B^0_{\infty,\infty}}ds
\\
&\lesssim\int_0^t (t-s)^{-1/2}\left(\|U^{(n)}\otimes U^{(n)}\|_{L^\infty}+\|V^{(n)}\otimes V^{(n)}\|_{L^\infty}\right)ds
\\
&\lesssim \left(\left(\underset{s<t}{\sup}\ \phi_1(s)\|U^{(n)}\|_{L^\infty}\right)^2 + \left(\underset{s<t}{\sup}\ \phi_1(s)\|V^{(n)}\|_{L^\infty}\right)^2\right)
\\
& \qquad\qquad\qquad \times \int_0^t (t-s)^{-1/2}(\phi_1(s))^{-2}ds\ .
\end{align*}
Thus
\begin{align}\label{eq:UnLinftyPsr}
\phi_1(t)\left\|\int_0^te^{(t-s)\Delta}\nabla\Pi^{(n)}ds\right\|_{L^\infty}&\lesssim t^{\frac{1}{2}}\psi_4(t)\left(\left(\underset{s<t}{\sup}\ \phi_1(s)\|U^{(n)}\|_{L^\infty}\right)^2 + \left(\underset{s<t}{\sup}\ \phi_1(s)\|V^{(n)}\|_{L^\infty}\right)^2\right)\ .
\end{align}
Following the same argument as in the estimation for $\|U^{(0)}\|_{L^\infty}$, we obtain
\begin{align}
&\left\|\int_0^t e^{(t-s)\Delta} \alpha\cdot \nabla V^{(n+1)} ds\right\|_{L^\infty} \lesssim |\alpha|t^{\frac{1}{2}} \psi_2(t) \| V^{(n+1)} \|_{L^\infty}\ , \label{eq:UnLinftyVnL}
\\
&\left\|\int_0^t e^{(t-s)\Delta} F\ ds\right\|_{L^\infty} \lesssim t\psi_1(t)\ \underset{s<t}{\sup}\ \|F \|_{bmo}\ ,\quad \|e^{t\Delta}u_0\|_{L^\infty}\lesssim \phi_1(t)^{-1}\|u_0\|_{bmo}\ . \label{eq:UnLinftyForce}
\end{align}
In summation, \eqref{eq:IterationU}, together with \eqref{eq:UnLinftyUnNL}-\eqref{eq:UnLinftyForce}, shows
\begin{align}\label{eq:UnLinfty}
\underset{t<T}{\sup}\ \phi_1(t)\|U^{(n)}\|_{L^\infty}\lesssim \|u_0\|_{bmo} + T^{\frac{1}{2}}\psi_4(T)\ (L_n)^2 + T\psi_3(T)\ \underset{t<T}{\sup}\|F\|_{bmo} + |\alpha|t^{\frac{1}{2}} \psi_2(t) L_{n+1}\ .
\end{align}
The estimations in ``$\phi_1(t)L^\infty$'' are similar but with special attention to ``the mixed terms'': With the induction hypothesis $V^{(n)}\in \dot B^1_{\infty,1}$ we have
\begin{align*}
\left\|\int_0^t e^{(t-s)\Delta}(U^{(n)}\cdot \nabla)V^{(n)}ds\right\|_{L^\infty}&\lesssim \int_0^t \|U^{(n)}\|_{L^\infty} \|e^{(t-s)\Delta} |\nabla V^{(n)}|\|_{L^\infty}ds
\\
&\lesssim\int_0^t \|U^{(n)}\|_{L^\infty} \|\nabla V^{(n)}\|_{\dot B^0_{\infty,1}}ds
\\
&\lesssim\int_0^t \|U^{(n)}\|_{L^\infty}\|V^{(n)}\|_{\dot B^1_{\infty,1}}ds
\\
&\lesssim \left(\underset{s<t}{\sup}\ \phi_1(s)\|U^{(n)}\|_{L^\infty}\right)\left(\underset{s<t}{\sup}\ s^{\frac{1}{2}}\|V^{(n)}\|_{\dot B^1_{\infty,1}}\right)
\\
&\qquad\qquad\qquad\times\int_0^t s^{-\frac{1}{2}}\left(\phi_1(s)\right)^{-1}ds\ .
\end{align*}
Thus
\begin{align*}
\phi_1(t)\left\|\int_0^t e^{(t-s)\Delta}(U^{(n)}\cdot \nabla)V^{(n)}ds\right\|_{L^\infty}&\lesssim t^{\frac{1}{2}}\phi_1(t)\psi_2(t) L_n L_n'''\ .
\end{align*}
Similar to the deduction for \eqref{eq:UnLinftyPsr}, we have
\begin{align*}
\phi_1(t)\left\|\int_0^te^{(t-s)\Delta}\nabla R^{(n)}ds\right\|_{L^\infty}&\lesssim \phi_1(t)\int_0^t (t-s)^{-1/2}\left\|P\left(U^{(n)}\otimes V^{(n)}\right)\right\|_{\dot B^0_{\infty,\infty}}ds
\\
&\lesssim \phi_1(t)\int_0^t (t-s)^{-1/2}\left\|U^{(n)}\otimes V^{(n)}\right\|_{L_\infty}ds
\\
&\lesssim t^{\frac{1}{2}}\psi_4(t)\left(\underset{s<t}{\sup}\ \phi_1(s)\|U^{(n)}\|_{L^\infty}\right) \left(\underset{s<t}{\sup}\ \phi_1(s)\|V^{(n)}\|_{L^\infty}\right)\ .
\end{align*}
With the above estimates at hand, we have from \eqref{eq:IterationV} that
\begin{align}\label{eq:VnLinfty}
\underset{t<T}{\sup}\ \phi_1(t)\|V^{(n)}\|_{L^\infty}&\lesssim  T^{\frac{1}{2}}\psi_4(T)\ (L_n)^2 + T^{\frac{1}{2}}\psi_5(T) L_n L_n''' \notag
\\
&\qquad\quad + T\psi_3(T)\ \underset{t<T}{\sup}\|G\|_{bmo} + |\alpha|t^{\frac{1}{2}} \psi_2(t) L_{n+1}\ .
\end{align}
So, with $\alpha$ such that $C|\alpha|t^{1/2}\psi_2(t) <1/2$ for all $t<T$ with some constant $C$, \eqref{eq:UnLinfty} and \eqref{eq:VnLinfty} imply that
\begin{align}
L_{n+1} &\lesssim \|u_0\|_{bmo} + T\psi_3(T)\left(\underset{t<T}{\sup}\|F\|_{bmo}+\underset{t<T}{\sup}\|G\|_{bmo}\right) \notag
\\
&\qquad\qquad\qquad\quad + T^{\frac{1}{2}}\psi_4(T)\ (L_n)^2 + T^{\frac{1}{2}}\phi_1(T)\psi_2(T) L_n L_n''' \ . \label{eq:IterBddLinfty}
\end{align}
Recall that $L^\infty\hookrightarrow bmo \hookrightarrow B^0_{\infty,\infty}$, it follows immediately
\begin{align}
L_{n+1}' &\lesssim \|u_0\|_{bmo} + T\psi_1(T)\left(\underset{t<T}{\sup}\|F\|_{bmo}+\underset{t<T}{\sup}\|G\|_{bmo}\right) \notag
\\
&\qquad\qquad\qquad\quad + T^{\frac{1}{2}}\psi_4(T)\left(\phi_1(T)\right)^{-1} (L_n)^2+ T^{\frac{1}{2}}\psi_2(T) L_n L_n''' \ ,
\\
L_{n+1}'' &\lesssim \|u_0\|_{bmo} + T\psi_1(T)\left(\underset{t<T}{\sup}\|F\|_{bmo}+\underset{t<T}{\sup}\|G\|_{bmo}\right) \notag
\\
&\qquad\qquad\qquad\quad + T^{\frac{1}{2}}\psi_4(T)\left(\phi_1(T)\right)^{-1} (L_n)^2+ T^{\frac{1}{2}}\psi_2(T) L_n L_n''' \ .
\end{align}
For the estimations in $\dot B^1_{\infty,1}$, due to large amount of overlap with those previous arguments, it suffices to demonstrate the followings: By Lemma~\ref{le:BMOEmbedBesov} and the boundedness of the projection $P$ in $\dot B^0_{\infty,\infty}$ with the fact that $U^{(n)}, V^{(n)}$ are divergence free,
\begin{align*}
\left\|\int_0^te^{(t-s)\Delta}\nabla R^{(n)}ds\right\|_{\dot B^1_{\infty,1}}&\lesssim \int_0^t \left\|e^{(t-s)\Delta}\nabla P\left(U^{(n)}\otimes V^{(n)}\right)\right\|_{\dot B^1_{\infty,1}}ds
\\
&\lesssim \int_0^t (t-s)^{-1/2}\left\|P\left(\nabla(U^{(n)}\otimes V^{(n)})\right)\right\|_{\dot B^0_{\infty,\infty}}ds
\\
&\lesssim \int_0^t (t-s)^{-1/2}\left\|(U^{(n)}\cdot \nabla)V^{(n)}+(V^{(n)}\cdot \nabla)U^{(n)}\right\|_{\dot B^0_{\infty,\infty}}ds\ .
\end{align*}
Then, with the induction hypothesis $U^{(n)}, V^{(n)}\in \dot B^1_{\infty,1}$, it follows that
\begin{align*}
\left\|\int_0^te^{(t-s)\Delta}\nabla R^{(n)}ds\right\|_{\dot B^1_{\infty,1}}
&\lesssim \int_0^t (t-s)^{-\frac{1}{2}}\left(\|U^{(n)}\|_{L^\infty}\|\nabla V^{(n)}\|_{L^\infty} + \|V^{(n)}\|_{L^\infty}\|\nabla U^{(n)}\|_{L^\infty}\right)ds
\\
&\lesssim \int_0^t (t-s)^{-\frac{1}{2}}\left(\|U^{(n)}\|_{L^\infty}\|\nabla V^{(n)}\|_{\dot B^0_{\infty,1}} + \|V^{(n)}\|_{L^\infty}\|\nabla U^{(n)}\|_{\dot B^0_{\infty,1}}\right)ds
\\
&\lesssim \psi_5(t)\left(\underset{s<t}{\sup}\ \phi_1(s)\|U^{(n)}\|_{L^\infty}\right) \left(\underset{s<t}{\sup}\ s^{\frac{1}{2}}\|V^{(n)}\|_{L^\infty}\right)
\\
&\qquad\qquad\qquad + \psi_5(t)\left(\underset{s<t}{\sup}\ \phi_1(s)\|V^{(n)}\|_{L^\infty}\right) \left(\underset{s<t}{\sup}\ s^{\frac{1}{2}}\|U^{(n)}\|_{L^\infty}\right)\
\end{align*}
where
\begin{align*}
\psi_5(t)=\int_0^t s^{-\frac{1}{2}} (t-s)^{-\frac{1}{2}} \left(\phi_1(s)\right)^{-1}ds\ .
\end{align*}
The bounds in $\dot B^1_{\infty,1}$ for all other nonlinear terms follow in a similar fashion. So
\begin{align}\label{eq:IterBddBinfty1}
L_{n+1}'''\lesssim_\epsilon \|u_0\|_{bmo} + T \left(\underset{t<T}{\sup}\|F\|_{bmo}+\underset{t<T}{\sup}\|G\|_{bmo}\right) + T^{\frac{1}{2}}\psi_5(T)\ L_n L_n'''\ .
\end{align}
In conclusion, \eqref{eq:IterBddLinfty}-\eqref{eq:IterBddBinfty1} yield uniform bound for all the function spaces:
\begin{align*}
M_{n+1}&:=\max\{L_{n+1},L_{n+1}',L_{n+1}'',L_{n+1}'''\}
\\
&\lesssim \|u_0\|_{bmo} + T\Psi_1(T) \left(\underset{t<T}{\sup}\|F\|_{bmo}+\underset{t<T}{\sup}\|G\|_{bmo}\right) + T^{\frac{1}{2}}\Psi_2(T)  (M_n)^2
\end{align*}
where
\begin{align*}
\Psi_1(t)=\max\{1,\psi_1(t),\psi_3(t)\}\ ,\qquad \Psi_2(t)=\max\{\psi_2,\psi_4,\psi_5,\phi_1\psi_2,\phi_1^{-1}\psi_4\}\ .
\end{align*}
If we take $T_*$ such that
\begin{align*}
T_*^{\frac{1}{2}}\Psi_2(T_*)\le \frac{1}{C\left(\|u_0\|_{bmo} + T_*\Psi_1(T_*) \left(\underset{t<T_*}{\sup}\|F\|_{bmo}+\underset{t<T_*}{\sup}\|G\|_{bmo}\right)\right)}
\end{align*}
where the constant $C$ was generated in our iteration scheme, independent of $T$, $u_0$, $F$ and $G$. Then all the sequences are bounded by $\left(2C T_*^{\frac{1}{2}}\Psi_2(T_*)\right)^{-1}$. This completes the proof of the claim.

Now the standard converging argument with Lemma~\ref{le:Montel} (applied for each $t$ with $p=\infty$) shall complete the proof that the limit function $u$ (i.e. the complexified solution of the NSE \eqref{eq:NSE1}-\eqref{eq:NSE3}) exists and it is bounded locally uniformly in time (the time interval only depends on $\|u_0\|_{bmo}$, $\|F\|_{bmo}$ and $\|G\|_{bmo}$) and uniformly in $y$-variables over the complex domain
\begin{align*}
\cD_t=: \left\{(x,y)\in\bC^d\ \big|\ |y|\le \min\{ct^{1/2}\min\{1,\psi_2(t)^{-1}\},\delta_f\}\right\}\
\end{align*}
in any of the four spaces (over $x\in\bR^d$) with the upper bound only depending on $\|u_0\|_{bmo}$, $\|F\|_{bmo}$ and $\|G\|_{bmo}$.  The analyticity properties of $u$, i.e. the existence of the higher order space derivatives is justified by the uniform convergence on any compact subset of $\cD_t$, following from Lemma~\ref{le:Montel} (see \citet{Grujic1998} for more details). This ends the proof of the theorem.
\end{proof}

In fact, by the existence of space derivatives and using the equation \eqref{eq:NSE1} (again with uniform convergence on compact subset in the complex space), it follows that the classical derivatives in time also exist. Therefore, we have

\begin{corollary}
The solution $u(t)$ stated in Theorem~\ref{th:MainThm} is the classical solution of \eqref{eq:NSE1}-\eqref{eq:NSE3}.
\end{corollary}

\bigskip

The following result is a type of sharp $L^\infty$-estimates for the heat semigroup convolved with Calder\'on-Zygmund operators, as stated and proved in a more general setup, being prepared only for the proof of Theorem~\ref{th:MainThmVor} (as well as for some follow-up questions in future).
\begin{lemma}\label{le:HeatSmGpLfty}
Let $T$ be a Calder\'on-Zygmund operator (see \citet{Stein1993} for the definition of C.Z.O.) with symmetric kernel $K(\cdot,\cdot)$ satisfying
\begin{align}\label{eq:RadialCancel}
\int_{\bS^d}K(x,z)d\sigma(z)=\int_{\bS^d}K(z,y)d\sigma(z)=0\quad\textrm{ for all }x,y
\end{align}
where $\bS^d$ denotes the unit sphere centered at $x$ or $y$. And let $\Phi\in L\log L(\bR^d)$ be a non-negative, radially symmetric and radially decreasing function. Given $k>0$, there exists a number $T^*$ (which only depends on $k$) such that, for any $t<T^*$ and for any function $g$ such that $|g|\le \Phi$ and any $f\in L^\infty(\bR^d)\cap L^p(\bR^d)$ for some $1\le p<\infty$, we have
\begin{align}
\left\|g_t*|Tf|^k\right\|_{L^\infty}\lesssim_{\Phi,p,k} \Psi_k(t)\left(\|f\|_{L^\infty}^k+\|f\|_{L^p}^k+\|f\|_{L^\infty}^{k\alpha}\|f\|_{L^p}^{k(1-\alpha)}\right)
\end{align}
for some $\alpha$, where $g_t(x):=t^{-d}g(x/t)$ and $\Psi_k(t)$ grows logarithmically as $t\to0^+$.
\end{lemma}

\medskip

\begin{proof}[Proof of Theorem~\ref{th:MainThmVor}]
We construct an approximating sequence from \eqref{eq:VelVorForm} as follows:
\begin{align*}
\partial_t\omega^{(n)} -\Delta \omega^{(n)} &=\omega^{(n-1)}\nabla u^{(n-1)} - u^{(n-1)}\nabla\omega^{(n-1)},\qquad \omega^{(n)}(0,x)=\omega_0\ ,
\\
u_j^{(n-1)}(x,t) &=c\int_{\mathbb{R}^3} \epsilon_{j,k,\ell}\ \partial_{y_k}\frac{1}{|x-y|}\omega_\ell^{(n-1)}(y,t)dy\ .
\end{align*}
We let $u^{(n)}+iv^{(n)}$ and $\omega^{(n)}+i\zeta^{(n)}$ be the analytic extension of the approximating sequence and let
\begin{align*}
&U^{(n)}(x,t)=u^{(n)}(x,\alpha t,t)\ , &&W^{(n)}(x,t)=w^{(n)}(x,\alpha t,t)\ ,
\\
&V^{(n)}(x,t)=v^{(n)}(x,\alpha t,t)\ , &&Z^{(n)}(x,t)=\zeta^{(n)}(x,\alpha t,t)\ ,
\end{align*}
for which we have the iterations:
\begin{align*}
W^{(n+1)}(x,t) &= e^{t\Delta}\omega_0 +\int_0^t e^{(t-s)\Delta}W^{(n)}\nabla U^{(n)}ds -\int_0^t e^{(t-s)\Delta}Z^{(n)}\nabla V^{(n)}ds
\\
&\quad-\int_0^t e^{(t-s)\Delta}U^{(n)}\nabla W^{(n)}ds +\int_0^t e^{(t-s)\Delta}V^{(n)}\nabla Z^{(n)}ds +\int_0^t e^{(t-s)\Delta}\alpha\cdot\nabla Z^{(n+1)} ds
\\
Z^{(n+1)}(x,t) &= \int_0^t e^{(t-s)\Delta}Z^{(n)}\nabla U^{(n)}ds +\int_0^t e^{(t-s)\Delta}W^{(n)}\nabla V^{(n)}ds
\\
&\quad-\int_0^t e^{(t-s)\Delta}V^{(n)}\nabla W^{(n)}ds -\int_0^t e^{(t-s)\Delta}U^{(n)}\nabla Z^{(n)}ds -\int_0^t e^{(t-s)\Delta}\alpha\cdot\nabla W^{(n+1)} ds
\end{align*}
where
\begin{align}
U_j^{(n)}(x,t) &=c\int_{\mathbb{R}^3} \epsilon_{j,k,\ell}\ \partial_{y_k}\frac{1}{|x-y|}W_\ell^{(n)}(y,t)dy\ , \label{eq:BiotSU}
\\
V_j^{(n)}(x,t) &=c\int_{\mathbb{R}^3} \epsilon_{j,k,\ell}\ \partial_{y_k}\frac{1}{|x-y|} Z_\ell^{(n)}(y,t)dy\ . \label{eq:BiotSV}
\end{align}
Analogous to the proof of Theorem~\ref{th:MainThm} we have the statement: \textit{There exists $T$ such that for all $n$
\begin{align}
W^{(n)},Z^{(n)} \in C([0,T);bmo(\bR^3)^3)\ ,\quad W^{(n)},Z^{(n)}&\in C([0,T); L^p(\bR^3)^3)\ , \label{eq:WnZnBMOLp}
\\
\phi_1(t)W^{(n)},\phi_1(t)Z^{(n)},\phi_1(t)U^{(n)},\phi_1(t)V^{(n)} &\in C([0,T);L^\infty(\bR^3)^3)\ , \label{eq:WnZnLinfty}
\end{align}
where $\phi_1(t)$ is given in Theorem~\ref{th:MainThmVor}. Moreover
\begin{align*}
&\underset{t<T}{\sup}\ \phi_1(t)\|W^{(n)}\|_{L^\infty}+\underset{t<T}{\sup}\ \phi_1(t)\|Z^{(n)}\|_{L^\infty}<K_n\ ,\quad \underset{t<T}{\sup}\ \|W^{(n)}\|_{bmo}+\underset{t<T}{\sup}\ \|Z^{(n)}\|_{bmo} <K_n''\ ,
\\
&\underset{t<T}{\sup}\ \|W^{(n)}\|_{B^0_{\infty,\infty}}+\underset{t<T}{\sup}\ \|Z^{(n)}\|_{B^0_{\infty,\infty}} <K_n'\ , \qquad\underset{t<T}{\sup}\ \|W^{(n)}\|_{L^p}+\underset{t<T}{\sup}\ \|Z^{(n)}\|_{L^p}<K_n'''\ ,
\\
&\underset{t<T}{\sup}\ \phi_1(t)\|U^{(n)}\|_{L^\infty}+\underset{t<T}{\sup}\ \phi_1(t)\|V^{(n)}\|_{L^\infty}<Q_n\ ,
\end{align*}
where $K_n,...,K_n''',Q_n$ are all bounded by a constant determined by $\max\{\|\omega_0\|_{bmo},\|\omega_0\|_{L^p}\}$.}

\bigskip

\textit{Proof of the Claim:} The estimates for $W^{(0)}$ and $Z^{(0)}$ are very similar and easier compared to Theorem~\ref{th:MainThm}, so we get straight to the conclusion: With the choice of $\alpha$ such that $C|\alpha|t^{1/2}\psi_2(t) <1/2$ for all $t<T$ and some constant $C$, we have
\begin{align*}
\cM_0:=\max\{K_0,K_0',K_0'',K_0'''\}\lesssim \|\omega_0\|_{bmo} + \|\omega_0\|_{L^p}\ .
\end{align*}

The essence of binding the rest of $\cM_n$'s still lies in the $L^\infty$-estimation (and $L^p$-estimation) of the nonlinear terms.
\begin{align*}
\left\|\int_0^t e^{(t-s)\Delta}W^{(n)}\nabla U^{(n)}ds\right\|_{L^\infty}&\lesssim\int_0^t \|e^{(t-s)\Delta}W^{(n)}\nabla U^{(n)}\|_{L^\infty}ds
\\
&\lesssim\int_0^t \|W^{(n)}\|_{L^\infty}\left\|e^{(t-s)\Delta}|\nabla U^{(n)}|\right\|_{L^\infty}ds\ .
\end{align*}
Notice that the map
\begin{align*}
(Tf)_j(x,t):=c\ \nabla\int_{\mathbb{R}^3} \epsilon_{j,k,\ell}\ \partial_{y_k}\frac{1}{|x-y|}f_\ell(y,t)dy
\end{align*}
defines a $\cC$-$\cZ$ operator that satisfies the conditions in Lemma~\ref{le:HeatSmGpLfty}, so applying the Lemma
\begin{align*}
\left\|\int_0^t e^{(t-s)\Delta}W^{(n)}\nabla U^{(n)}ds\right\|_{L^\infty}&\lesssim \int_0^t \|W^{(n)}\|_{L^\infty} \left\|e^{(t-s)\Delta}|T W^{(n)}|\right\|_{L^\infty}ds
\\
&\lesssim \int_0^t \phi_1(t-s)^{-1}\|W^{(n)}\|_{L^\infty}\left(\|W^{(n)}\|_{L^\infty}+\|W^{(n)}\|_{L^p}\right)ds
\\
&\lesssim K_n\left(K_n+K_n'''\right)\int_0^t \phi_1(t-s)^{-1}\left(\phi_1(s)^{-1}+\phi_1(s)^{-2}\right)ds\ .
\end{align*}
Let $W_x^{(n)}(y)$ denote the translation $W^{(n)}(x-y)$ and $B$ be the unit ball centered at 0. Then, from \eqref{eq:BiotSU} we know
\begin{align*}
\left|U^{(n)}(x,t)\right| &\lesssim \int_B \frac{1}{|y|^2} \left|W_x^{(n)}(y,t)\right| dy + \int_{B^c} \frac{1}{|y|^2} \left|W_x^{(n)}(y,t)\right| dy
\\
&\lesssim \|W_x^{(n)}\|_{{L^\infty}}\int_B |y|^{-2}dy + \|W_x^{(n)}\|_{L^p}\||y|^{-2}\1_{B^c}\|_{L^{p'}}\lesssim \|W^{(n)}\|_{{L^\infty}} + \|W^{(n)}\|_{L^p}\ ,
\end{align*}
where we used the fact $p'>\frac{3}{2}$ (since $p<3$), so
\begin{align}\label{eq:VorUnLinfty}
\|U^{(n)}\|_{L^\infty} \lesssim \|W^{(n)}\|_{{L^\infty}} + \|W^{(n)}\|_{L^p}\ .
\end{align}
Similarly, $\|V^{(n)}\|_{L^\infty} \lesssim \|Z^{(n)}\|_{{L^\infty}} + \|Z^{(n)}\|_{L^p}$.
Then, by Lemma~\ref{le:HeatSmGpLfty} it follows that
\begin{align*}
\left|\int_0^t \nabla e^{(t-s)\Delta}U^{(n)}W^{(n)} ds\right|
&\lesssim\int_0^t \|W^{(n)}\|_{L^\infty}\left\|\nabla G_{t-s}(x-\cdot)U^{(n)}(\cdot)\right\|_{L^1}ds
\\
&\lesssim\int_0^t \|W^{(n)}\|_{L^\infty}\left(\left\|G_{t-s}(x-\cdot)|\nabla |U^{(n)}(\cdot)||\right\|_{L^1}+ 6|U^{(n)}(x)|\right)ds
\\
&\lesssim\int_0^t \|W^{(n)}\|_{L^\infty}\left(\sup_{x\in\bR^3}\left\|G_{t-s}(x-\cdot)|\nabla U^{(n)}(\cdot)|\right\|_{L^1} + \|U^{(n)}\|_{L^\infty}\right)ds
\\
&\lesssim\int_0^t \|W^{(n)}\|_{L^\infty}\left(\left\|e^{(t-s)\Delta}|\nabla U^{(n)}|\right\|_{L^\infty} + \|W^{(n)}\|_{{L^\infty}} + \|W^{(n)}\|_{L^p}\right)ds
\\
&\lesssim K_n\left(K_n+K_n'''\right)\int_0^t \left(1+\phi_1(t-s)^{-1}\right) \left(\phi_1(s)^{-1}+\phi_1(s)^{-2}\right)ds\ .
\end{align*}
Therefore
\begin{align*}
\left\|\int_0^t e^{(t-s)\Delta}U^{(n)}\nabla W^{(n)}ds\right\|_{L^\infty} &\lesssim \left\|\int_0^t e^{(t-s)\Delta}W^{(n)}\nabla U^{(n)}ds\right\|_{L^\infty} + \left\|\int_0^t \nabla e^{(t-s)\Delta}U^{(n)}W^{(n)} ds\right\|_{L^\infty}
\\
&\lesssim K_n\left(K_n+K_n'''\right)\int_0^t \left(1+\phi_1(t-s)^{-1}\right) \left(\phi_1(s)^{-1}+\phi_1(s)^{-2}\right)ds\ .
\end{align*}

For the estimations in $L^p$, we divide the proof into two cases. For $p>1$: by Young's inequality and the $L^p$-boundedness of $T$, we have
\begin{align*}
\left\|\int_0^t e^{(t-s)\Delta}W^{(n)}\nabla U^{(n)}ds\right\|_{L^p}&\lesssim \int_0^t \|W^{(n)}\|_{L^\infty} \left\|e^{(t-s)\Delta}|T W^{(n)}|\right\|_{L^p}ds
\\
&\lesssim \int_0^t \|W^{(n)}\|_{L^\infty}\|W^{(n)}\|_{L^p}ds \lesssim K_nK_n'''\int_0^t \phi_1(s)^{-1}ds\ .
\end{align*}
Similar to the argument for $L^\infty$-estimation, we deduce
\begin{align*}
\left\|\int_0^t \nabla e^{(t-s)\Delta}U^{(n)}W^{(n)} ds\right\|_{L^p}
&\lesssim \int_0^t \left\|\nabla e^{(t-s)\Delta}U^{(n)}W^{(n)}\right\|_{L^p} ds
\\
&\lesssim \int_0^t \left\|\nabla G_{t-s}\right\|_{L^1} \left\|U^{(n)}W^{(n)}\right\|_{L^p} ds
\\
&\lesssim \int_0^t (t-s)^{-1/2} \|U^{(n)}\|_{L^\infty}\|W^{(n)}\|_{L^p}\ ds
\\
&\lesssim Q_nK_n'''\int_0^t (t-s)^{-1/2}\phi_1(s)^{-1} ds\ .
\end{align*}
Combining the above two results,
\begin{align*}
\left\|\int_0^t e^{(t-s)\Delta}U^{(n)}\nabla W^{(n)}ds\right\|_{L^p} &\lesssim K_n K_n''' \int_0^t \phi_1(s)^{-1}ds + Q_nK_n'''\int_0^t (t-s)^{-1/2}\phi_1(s)^{-1} ds\ .
\end{align*}
For $p=1$: by Young's inequality, the $L^p$-boundedness of $T$ and interpolations in $L^p$, we have
\begin{align*}
\left\|\int_0^t e^{(t-s)\Delta}W^{(n)}\nabla U^{(n)}ds\right\|_{L^1} &\lesssim \int_0^t \left\| e^{(t-s)\Delta}W^{(n)}\nabla U^{(n)}\right\|_{L^1}ds
\\
&\lesssim \int_0^t \|W^{(n)}\|_{L^2}\|\nabla U^{(n)}\|_{L^2}ds
\\
&\lesssim \int_0^t \|W^{(n)}\|_{L^\infty}\|W^{(n)}\|_{L^1}ds \lesssim K_nK_n'''\int_0^t \phi_1(s)^{-1}ds\
\end{align*}
Similar to the case $p>1$,
\begin{align*}
\left\|\int_0^t \nabla e^{(t-s)\Delta}U^{(n)}W^{(n)}ds\right\|_{L^1}&\lesssim Q_nK_n'''\int_0^t (t-s)^{-1/2}\phi_1(s)^{-1} ds\ .
\end{align*}
Combining the above two results,
\begin{align*}
\left\|\int_0^t e^{(t-s)\Delta}U^{(n)}\nabla W^{(n)}ds\right\|_{L^1} &\lesssim K_n'''(K_n+Q_n) \int_0^t \left(\phi_1(s)^{-1} + (t-s)^{-1/2}\phi_1(s)^{-1}\right)ds\ .
\end{align*}

After estimating all the other nonlinear terms in exactly the same way, one can conclude that: If $1\le p<3$, then
\begin{align*}
\cM_{n+1}&:=\max\{K_{n+1},K_{n+1}',K_{n+1}'',K_{n+1}''', Q_{n+1}\}
\\
&\lesssim \|\omega_0\|_{bmo}+\|\omega_0\|_{L^p} + T^{1/2}\Psi_1^\omega(T)\left(\cM_n\right)^2\ ,
\end{align*}
where $\Psi_1^\omega$ is some function with logarithmic blow-up at $t=0$.
If we take $T_\omega$ such that
\begin{align*}
T_\omega^{1/2}\Psi_1^\omega(T_\omega)\lesssim \left(\|\omega_0\|_{bmo}+\|\omega_0\|_{L^p}\right)^{-1}\ ,
\end{align*}
then all the sequences are bounded by $\displaystyle\left(T_\omega^{1/2}\Psi_1^\omega(T_\omega)\right)^{-1}$. This ends the proof of the claim.

The rest of the proof is similar to that of Theorem~\ref{th:MainThm}, we omit the details.
\end{proof}

\bigskip

\begin{proof}[Proof of Theorem~\ref{th:MainThmVor2}] By Theorem~\ref{th:MainThmVor} there exists $t_\epsilon \ge (\|\omega_0\|_{bmo}+\|\omega_0\|_{L^p})^{-\delta}$ with $\delta>2$ such that for some constant $M_0\approx 1$
\begin{align*}
\|\omega(t_\epsilon)\|_{bmo} &\le M_0\ \|\omega_0\|_{bmo}\ ,\qquad\ \|\omega(t_\epsilon)\|_{L^p} \le M_0\ \|\omega_0\|_{L^p}\ ,
\\
\|\omega(t_\epsilon)\|_{L^\infty} &\le M_0\ \phi_1(t_\epsilon)^{-1} \|\omega_0\|_{bmo} \le \delta M_0\ \|\omega_0\|_{bmo}\ \phi_1(\|\omega_0\|_{bmo}+\|\omega_0\|_{L^p})^{-1}\ .
\end{align*}
By $L^p$-interpolation, for any pair $(q,r)$ such that $p<r<3$ and $1/q\le 1/r-1/3$,
\begin{align*}
\|\omega(t_\epsilon)\|_{L^r} &\le \|\omega(t_\epsilon)\|_{L^p}^{\frac{p}{r}}\|\omega(t_\epsilon)\|_{L^\infty}^{1-\frac{p}{r}} \le \delta M_0\ \|\omega_0\|_{L^p}^{\frac{p}{r}} \left(\|\omega_0\|_{bmo}\ \phi_1(\|\omega_0\|_{bmo}+\|\omega_0\|_{L^p})^{-1}\right)^{1-\frac{p}{r}} \ ,
\\
\|\omega(t_\epsilon)\|_{L^q} &\le \delta M_0\ \|\omega_0\|_{L^p}^{\frac{p}{q}} \left(\|\omega_0\|_{bmo}\ \phi_1(\|\omega_0\|_{bmo}+\|\omega_0\|_{L^p})^{-1}\right)^{1-\frac{p}{q}}
\end{align*}
and by the $L^p$-bound of $\mathcal{CZO}$ and the estimate in \eqref{eq:VorUnLinfty},
\begin{align*}
\|u(t_\epsilon)\|_{L^q} &\lesssim \|\nabla u(t_\epsilon)\|_{L^r}^s \|u(t_\epsilon)\|_{L^\infty}^{1-s} \lesssim \|\omega(t_\epsilon)\|_{L^r}^s \left(\|\omega(t_\epsilon)\|_{L^\infty} + \|\omega(t_\epsilon)\|_{L^p} \right)^{1-s}
\\
&\lesssim M_0 \ell\  \|\omega_0\|_{L^p}^{\frac{ps}{r}} \left(\|\omega_0\|_{bmo}\ \phi_1(\|\omega_0\|_{bmo}+\|\omega_0\|_{L^p})^{-1} + \|\omega_0\|_{L^p}\right)^{s-\frac{ps}{r}}
\\
&\qquad\qquad\qquad\qquad\times \left(\|\omega_0\|_{bmo}\ \phi_1(\|\omega_0\|_{bmo}+\|\omega_0\|_{L^p})^{-1} + \|\omega_0\|_{L^p} \right)^{1-s}
\end{align*}
where $s=\displaystyle \frac{r}{(1-\frac{r}{3})q}$.
If choose $t_\epsilon \approx (\|\omega_0\|_{bmo}+\|\omega_0\|_{L^p})^{-\delta}$, then by the estimates for $\left|U^{(n)}(x,t)\right|$ and $\left|V^{(n)}(x,t)\right|$ in the proof of Theorem~\ref{th:MainThmVor} we know
\begin{align*}
\underset{s<t_\epsilon}{\sup}\ \phi_1(s)\|U^{(n)}\|_{L^\infty} & + \underset{s<t_\epsilon}{\sup}\ \phi_1(s)\|V^{(n)}\|_{L^\infty} \le \underset{s<t_\epsilon}{\sup}\ \phi_1(s)\|W^{(n)}\|_{L^\infty}+\underset{s<t_\epsilon}{\sup}\ \phi_1(s)\|Z^{(n)}\|_{L^\infty}
\\
&\quad + \underset{s<t_\epsilon}{\sup}\ \phi_1(s)\|W^{(n)}\|_{L^p}+\underset{s<t_\epsilon}{\sup}\ \phi_1(s)\|Z^{(n)}\|_{L^p} \le M_0\ \|\omega_0\|_{bmo} + \phi_1(t_\epsilon)\|\omega_0\|_{L^p}\ .
\end{align*}
Thus, a standard converging argument yields
\begin{align*}
\underset{s\in(0,t_\epsilon)}{\sup}\ \underset{y\in\cD_s}{\sup}\ \phi_1(s)\|u(\cdot,y,s)\|_{L^\infty} + \underset{s\in(0,t_\epsilon)}{\sup}\ \underset{y\in\cD_s}{\sup}\ \phi_1(s)\|v(\cdot,y,s)\|_{L^\infty}\le M_0\ \|\omega_0\|_{bmo} \
\end{align*}
where
\begin{align*}
\cD_s=: \left\{(x,y)\in\bC^3\ \big|\ |y|\lesssim (1-M_0^{-1})s^{1/2}\Phi_2(s) \right\} \ .
\end{align*}
In particular, $\|u(t_\epsilon)\|_{L^\infty} \le M_0\ \phi_1(t_\epsilon)^{-1} \|\omega_0\|_{bmo} \le \delta M_0\ \|\omega_0\|_{bmo}\ \phi_1(\|\omega_0\|_{bmo}+\|\omega_0\|_{L^p})^{-1}$.
Let $\bar K_n$, $\bar K_n''$, $\bar K_n'''$, $\bar Q_n$ and $\bar Q_n'$ be such that $\displaystyle \underset{t_\epsilon<s<T}{\sup}\left( \|W^{(n)}\|_{bmo}+ \|Z^{(n)}\|_{bmo}\right)<\bar K_n''$,
\begin{align*}
\underset{t_\epsilon<s<T}{\sup} \phi_1(s-t_\epsilon) \left( \|W^{(n)}\|_{L^\infty} + \|Z^{(n)}\|_{L^\infty}\right)<\bar K_n\ ,\qquad \underset{t_\epsilon<s<T}{\sup}\left( \|W^{(n)}\|_{L^q}+ \|Z^{(n)}\|_{L^q}\right)<\bar K_n''' \ ,
\\
\underset{t_\epsilon<s<T}{\sup} \phi_1(s-t_\epsilon) \left( \|U^{(n)}\|_{L^\infty}+ \|V^{(n)}\|_{L^\infty}\right)<\bar Q_n\ , \qquad \underset{t_\epsilon<s<T}{\sup}\left( \|U^{(n)}\|_{L^q}+ \|V^{(n)}\|_{L^q}\right)<\bar Q_n'\ .
\end{align*}
The estimates for $U^{(0)},V^{(0)},W^{(0)},Z^{(0)}$ initiated at $t_\epsilon$ are the same as that for Theorem~\ref{th:MainThm} and Theorem~\ref{th:MainThmVor}: With the choice of $\alpha$ such that $|\alpha|(t-t_\epsilon)^{1/2}\lesssim 1-M_\epsilon^{-1}$ for all $t-t_\epsilon<T$, we have
\begin{align*}
\bar \cM_0&:=\max\{\bar K_0,\bar K_0'',\bar K_0''',\bar Q_0,\bar Q_0'\}\lesssim \|u(t_\epsilon)\|_{L^\infty} + \|u(t_\epsilon)\|_{L^q} + \|\omega(t_\epsilon)\|_{L^\infty} + \|\omega(t_\epsilon)\|_{L^q}
\\
&\lesssim \left(\|\omega_0\|_{bmo} + \|\omega_0\|_{L^p} + \|\omega_0\|_{L^p}^{\frac{p}{q}} \|\omega_0\|_{bmo}^{1-\frac{p}{q}} + \|\omega_0\|_{L^p}^{\frac{ps}{r}}\|\omega_0\|_{bmo}^{s-\frac{ps}{r}}\right) \tilde\Phi\left(\|\omega_0\|_{bmo}+\|\omega_0\|_{L^p}\right)
\\
&\lesssim \delta M_0\ \left(\|\omega_0\|_{bmo} + \|\omega_0\|_{L^p}\right) \tilde\Phi\left(\|\omega_0\|_{bmo}+\|\omega_0\|_{L^p}\right)
\end{align*}
where $\tilde \Phi$ is some function with logarithmic growth as $r\to\infty$.
The $L^\infty$-estimates for the nonlinear terms are summarized as
\begin{align*}
\left\|\int_{t_\epsilon}^t e^{(t-s)\Delta}W^{(n)}\nabla U^{(n)}ds\right\|_{L^\infty}&\lesssim \int_{t_\epsilon}^t \phi_1(t-s)^{-1}\|W^{(n)}\|_{L^\infty}\left(\|W^{(n)}\|_{L^\infty}+\|W^{(n)}\|_{L^q}\right)ds \ .
\end{align*}
and
\begin{align*}
\left\|\int_{t_\epsilon}^t \nabla e^{(t-s)\Delta}U^{(n)}W^{(n)} ds\right\|_{L^\infty} &\lesssim\int_{t_\epsilon}^t \left(1+\phi_1(t-s)^{-1}\right)\|W^{(n)}\|_{L^\infty}\left(\|W^{(n)}\|_{{L^\infty}} + \|W^{(n)}\|_{L^q}\right)ds
\end{align*}
The $L^q$-estimates are summarized as
\begin{align*}
\left\|\int_{t_\epsilon}^t e^{(t-s)\Delta}\nabla U^{(n)}W^{(n)} ds\right\|_{L^q}
&\lesssim \int_{t_\epsilon}^t \|W^{(n)}\|_{L^\infty}\|W^{(n)}\|_{L^q}ds
\end{align*}
and
\begin{align*}
\left\|\int_{t_\epsilon}^t \nabla e^{(t-s)\Delta}U^{(n)}W^{(n)} ds\right\|_{L^q}
&\lesssim\int_{t_\epsilon}^t \|W^{(n)}\|_{L^\infty}\left\||\nabla G_{t-s}|*|U^{(n)}|\right\|_{L^q}ds
\\
&\lesssim\int_{t_\epsilon}^t \|W^{(n)}\|_{L^\infty}\left(\left\|G_{t-s}*|\nabla|U^{(n)}||\right\|_{L^q}+ 6\|U^{(n)}\|_{L^q}\right)ds
\\
&\lesssim\int_{t_\epsilon}^t \|W^{(n)}\|_{L^\infty}\left(\|W^{(n)}\|_{L^q}+ \|U^{(n)}\|_{L^q}\right)ds \ .
\end{align*}
Applying Lemma~\ref{le:HeatSmGpLfty} twice to the `pressure term' yields
\begin{align*}
\left\|\int_{t_\epsilon}^te^{(t-s)\Delta}\nabla\Pi^{(n)}ds\right\|_{L^\infty}&\lesssim\int_{t_\epsilon}^t \left\| e^{(t-s)\Delta} P\left((U^{(n)}\cdot \nabla)U^{(n)}-(V^{(n)}\cdot \nabla)V^{(n)}\right)\right\|_{L^\infty}ds
\\
&\lesssim\int_{t_\epsilon}^t \phi_1(t-s)^{-2}\left(\left\|e^{\frac{t-s}{2}\Delta}\left((U^{(n)}\cdot \nabla)U^{(n)}-(V^{(n)}\cdot \nabla)V^{(n)}\right)\right\|_{L^q} \right.
\\
&\qquad\qquad\qquad\quad \left. +  \left\|e^{\frac{t-s}{2}\Delta}\left((U^{(n)}\cdot \nabla)U^{(n)}-(V^{(n)}\cdot \nabla)V^{(n)}\right)\right\|_{L^\infty}\right)ds
\\
&\lesssim\int_{t_\epsilon}^t \phi_1(t-s)^{-2} \left(\|U^{(n)}\|_{L^\infty}\|\nabla U^{(n)}\|_{L^q} + \|V^{(n)}\|_{L^\infty}\|\nabla V^{(n)}\|_{L^q} \right.
\\
&\quad \left. + \|U^{(n)}\|_{L^\infty}\left\|e^{\frac{t-s}{2}\Delta}|T(W^{(n)})|\right\|_{L^\infty}  + \|V^{(n)}\|_{L^\infty}\left\|e^{\frac{t-s}{2}\Delta} |T(Z^{(n)})|\right\|_{L^\infty} \right)ds
\\
&\lesssim\int_{t_\epsilon}^t \phi_1(t-s)^{-2} \left(\|U^{(n)}\|_{L^\infty}(\|W^{(n)}\|_{L^\infty} + \|W^{(n)}\|_{L^q}) \right.
\\
&\qquad\qquad\qquad\qquad\qquad \left.  + \|V^{(n)}\|_{L^\infty}(\|Z^{(n)}\|_{L^\infty} + \|Z^{(n)}\|_{L^q}) \right) ds
\end{align*}
Similarly, by the $L^q$-bound of $P$ and $\mathcal{CZO}$,
\begin{align*}
\left\|\int_{t_\epsilon}^te^{(t-s)\Delta}\nabla\Pi^{(n)}ds\right\|_{L^q} &\lesssim \int_{t_\epsilon}^t \left\| e^{(t-s)\Delta} P\left((U^{(n)}\cdot \nabla)U^{(n)}-(V^{(n)}\cdot \nabla)V^{(n)}\right)\right\|_{L^q}ds
\\
&\lesssim \int_{t_\epsilon}^t \left(\|(U^{(n)}\cdot \nabla)U^{(n)}\|_{L^q} + \|(V^{(n)}\cdot \nabla)V^{(n)}\|_{L^q}\right) ds
\\
&\lesssim (t-t_\epsilon) \left(\|U^{(n)}\|_{L^\infty}\|W^{(n)}\|_{L^q} + \|V^{(n)}\|_{L^\infty}\|Z^{(n)}\|_{L^q} \right)
\end{align*}
Define $\bar \cM_n:=\max\{\bar K_n,\bar K_n'',\bar K_n''',\bar Q_n,\bar Q_n'\}$. Following the same deduction as in the proofs of Theorem~\ref{th:MainThmVor} and Theorem~\ref{th:MainThm}, with the choice of $\alpha$ such that $C|\alpha|(t-t_\epsilon)^{1/2}\lesssim 1-M_\epsilon^{-1}$,
\begin{align*}
\|W^{(n+1)}\|_{bmo} + \|Z^{(n+1)}\|_{bmo} &\lesssim \|\omega(t_\epsilon)\|_{bmo} + (T-t_\epsilon)\psi_1(T)\ \bar \cM_n \bar K_n''
\\
\|W^{(n+1)}\|_{L^q} + \|Z^{(n+1)}\|_{L^q} &\lesssim \|\omega(t_\epsilon)\|_{L^q} + (T-t_\epsilon)\psi_2(T)\ \bar \cM_n \bar K_n'''
\\
\|U^{(n+1)}\|_{L^q} + \|V^{(n+1)}\|_{L^q} &\lesssim \|u(t_\epsilon)\|_{L^q} + (T-t_\epsilon)\psi_3(T)\ \bar \cM_n \bar Q_n'
\\
\phi_1(s-t_\epsilon) \left(\|W^{(n+1)}\|_{L^\infty} + \|Z^{(n+1)}\|_{L^\infty}\right) &\lesssim \|\omega(t_\epsilon)\|_{bmo} + (T-t_\epsilon)\psi_4(T)\ \bar \cM_n \bar K_n
\\
\|U^{(n+1)}\|_{L^\infty} + \|V^{(n+1)}\|_{L^\infty} &\lesssim \|u(t_\epsilon)\|_{L^\infty} + (T-t_\epsilon)\psi_5(T)\ \bar \cM_n \bar Q_n
\end{align*}
So, if $(T-t_\epsilon)\Psi(T)\lesssim \left(\|u(t_\epsilon)\|_{L^\infty} + \|u(t_\epsilon)\|_{L^q} + \|\omega(t_\epsilon)\|_{bmo} + \|\omega(t_\epsilon)\|_{L^q}\right)^{-1}$, then
\begin{align*}
\underset{s\in(t_\epsilon,T)}{\sup}\ \underset{y\in\bar \cD_{s-t_\epsilon}}{\sup} \phi_1(s-t_\epsilon) \|\omega(\cdot,y,s)\|_{L^\infty} + \underset{s\in(t_\epsilon,T)}{\sup}\ \underset{y\in\bar \cD_{s-t_\epsilon}}{\sup} \phi_1(s-t_\epsilon) \|\zeta(\cdot,y,s)\|_{L^\infty} &\le M_\epsilon \|\omega(t_\epsilon)\|_{bmo}
\\
\underset{s\in(t_\epsilon,T)}{\sup}\ \underset{y\in\bar \cD_{s-t_\epsilon}}{\sup} \|u(\cdot,y,s)\|_{L^\infty} + \underset{s\in(t_\epsilon,T)}{\sup}\ \underset{y\in\bar \cD_{s-t_\epsilon}}{\sup} \|v(\cdot,y,s)\|_{L^\infty} &\le M_\epsilon \|u(t_\epsilon)\|_{L^\infty}
\\
\underset{s\in(t_\epsilon,T)}{\sup}\ \underset{y\in\bar \cD_{s-t_\epsilon}}{\sup} \|\omega(\cdot,y,s)\|_{bmo} + \underset{s\in(t_\epsilon,T)}{\sup}\ \underset{y\in\bar \cD_{s-t_\epsilon}}{\sup} \|\zeta(\cdot,y,s)\|_{bmo} &\le M_\epsilon \|\omega(t_\epsilon)\|_{bmo}
\end{align*}
where
\begin{align*}
\bar\cD_s=: \left\{(x,y)\in\bC^3\ \big|\ |y|\lesssim (1-M_\epsilon^{-1})s^{1/2}\bar\Phi_2(s) \right\} \ .
\end{align*}
Thus, if $(T-t_\epsilon)\Psi(T)\lesssim (\delta M_0)^{-1} \left(\|\omega_0\|_{bmo} + \|\omega_0\|_{L^p}\right)^{-1} \tilde\Phi\left(\|\omega_0\|_{bmo}+\|\omega_0\|_{L^p}\right)^{-1}$, combining the results of the periods before and after $t_\epsilon$ gives
\begin{align*}
\underset{t\in(0,T)}{\sup}\ \underset{y\in\tilde \cD_t}{\sup}\|\omega(\cdot,y,t)\|_{bmo} + \underset{t\in(0,T)}{\sup}\ \underset{y\in\tilde \cD_t}{\sup}\|\zeta(\cdot,y,t)\|_{bmo} &\le \max\{ M_\epsilon \|\omega(t_\epsilon)\|_{bmo}\ , M_0 \|\omega_0\|_{bmo}\}
\\
&\le M_\epsilon M_0 \|\omega_0\|_{bmo}
\end{align*}
where $\tilde\cD_s$ is determined such that $\cup_{0<s<T}\tilde\cD_s=(\cup_{0<s<T_\omega}\cD_s)\cup(\cup_{t_\epsilon<s<T} \bar \cD_{s-t_\epsilon})$ where $T_\omega$ is given in Theorem~\ref{th:MainThmVor}.
Note that $T$ depends on $t_\epsilon$, so does the size of $\tilde\cD_t$ at $T$. For a fixed increment of $\sup_{y\in\tilde \cD_t} \|\omega(\cdot,y,t)\|_{bmo}$, i.e. $M_\epsilon M_0$ being fixed, to maximize the analyticity radius at $T$, we pick $\delta\approx 2$ while $M_0\approx 1$.
Now recall \eqref{eq:VorUnLinfty} and for $t_\epsilon <s< T$, with $|\alpha|(t-t_\epsilon)^{1/2}\lesssim 1-M_\epsilon^{-1}$,
\begin{align*}
\phi_1(s-t_\epsilon) \left(\|U^{(n+1)}\|_{L^\infty} + \|V^{(n+1)}\|_{L^\infty}\right) &\le \phi_1(s-t_\epsilon) \left(\|W^{(n+1)}\|_{L^\infty} + \|Z^{(n+1)}\|_{L^\infty}\right)
\\
&\qquad + \phi_1(s-t_\epsilon) \left(\|W^{(n+1)}\|_{L^q} + \|Z^{(n+1)}\|_{L^q}\right)
\\
\le M_\epsilon\ \|\omega(t_\epsilon)\|_{bmo} & +  \phi_1(T-t_\epsilon)  \|\omega(t_\epsilon)\|_{L^q} + (T-t_\epsilon)\tilde \psi_5(T)\ \bar \cM_n \bar Q_n
\end{align*}
So if $(T-t_\epsilon)\Psi(T)\lesssim \left(\|u(t_\epsilon)\|_{L^q} + \|\omega(t_\epsilon)\|_{bmo} + \|\omega(t_\epsilon)\|_{L^q}\right)^{-1}$, then
\begin{align*}
\underset{s\in(T_\omega,T)}{\sup}\ \underset{y\in\bar \cD_{s-t_\epsilon}}{\sup} \phi_1(s-t_\epsilon) \|u(\cdot,y,s)\|_{L^\infty} + \underset{s\in(T_\omega,T)}{\sup}\ \underset{y\in\bar \cD_{s-t_\epsilon}}{\sup} \phi_1(s-t_\epsilon) \|v(\cdot,y,s)\|_{L^\infty} &\le M_\epsilon \|\omega(t_\epsilon)\|_{bmo}
\end{align*}
This, combined with the result in Theorem~\ref{th:MainThmVor}, leads to
\begin{align*}
&\underset{t\in(0,T)}{\sup}\ \underset{y\in\tilde \cD_t}{\sup}\ \phi_1(t) \|u(\cdot,y,t)\|_{L^\infty} + \underset{t\in(0,T)}{\sup}\ \underset{y\in\tilde \cD_t}{\sup}\ \phi_1(t) \|v(\cdot,y,t)\|_{L^\infty}
\\
&\quad \le \max\{M\ \|\omega_0\|_{bmo}\ , \phi_1(T_\omega-t_\epsilon)^{-1}\phi_1(T_\omega) M_\epsilon\ \|\omega(t_\epsilon)\|_{bmo}\} \le \max\{M, M_\epsilon M_0\}\ \|\omega_0\|_{bmo}
\end{align*}
\end{proof}

\begin{appendix}

\section{Appendix}\label{sec:Appedix}

In this appendix, we first take a brief survey of some auxiliary while necessary results on the heat semigroup in $BMO$-type norm, H\"older type estimate in Besov spaces and equivalency and embedding among some oscillation function spaces. Then we prove the technical lemmas we quoted in Section~\ref{sec:MainResult}.

The following result shows the heat semigroup is bounded in $BMO(\bR^d)$. The idea for proving is to define the norm of Hardy space $\cH^1$ through the convolution with the heat kernel and use the duality by \citet{Fefferman1972}. For the detailed definitions of $BMO$ and $\cH^1$, the reader may refer to \citet{Stein1993}
\begin{theorem}[\citet{Bolkart2018}]\label{th:semigpBMO}
Consider the equation $\partial_t u-\Delta u =0 $ in $\bR^d\times [0,\infty)$ with $u(0) =u_0\in BMO(\bR^d)$. There is a solution $u(t)$ and constant $C$ satisfying the estimate
\begin{align}
\underset{t>0}{\sup}\left(\|u(t)\|_{BMO} +t^{\frac{1}{2}}\|\nabla u(t)\|_{L^\infty} + t\|\nabla^2 u(t)\|_{L^\infty} + t\|\partial_t u\|_{L^\infty}\right)\le C\|u_0\|_{BMO}\ .
\end{align}
\end{theorem}

\bigskip

The following two lemmas exhibit several H\"older-type estimates for the heat semigroup in $BMO$ and the Besov spaces. For the detailed definitions of Besov spaces, the reader may refer to \citet{Bahouri2011}, \citet{Triebel2010} and \citet{Kozono2003}.
\begin{lemma}[\citet{Giga1999} and \citet{Kozono2003}]\label{le:SemGpBMO}
For all $f\in BMO(\bR^d)$ we have the estimate
\begin{align}
\left\|(-\Delta)^\alpha e^{t\Delta}f\right\|_{L^\infty}\lesssim t^{-\alpha}\|f\|_{BMO}\ .
\end{align}
where $\alpha>0$.
\end{lemma}
The proof by \citet{Giga1999} is based on the estimate for the maximal function of $(-\Delta)^\alpha e^{t\Delta}f$. The idea by \citet{Kozono2003} is to compute $\|(-\Delta)^\alpha G_t\|_{\cH^1}$ ($G_t$ is the heat kernel) and use $\cH^1$-$BMO$ duality.

\begin{lemma}[\citet{Kozono2003}]\label{le:HolderEstBesov}
If $s_0\le s_1$, $1\le p,q\le \infty$, then there holds
\begin{align}
\|e^{t\Delta}f\|_{\dot B_{p,q}^{s_1}} &\lesssim t^{-\frac{1}{2}(s_1-s_0)}\|f\|_{\dot B_{p,q}^{s_0}} \label{eq:HeatHomBesov}
\\
\|e^{t\Delta}f\|_{B_{p,q}^{s_1}} &\lesssim \left(1+t^{-\frac{1}{2}(s_1-s_0)}\right)\|f\|_{B_{p,q}^{s_0}} \label{eq:HeatNonHomBesov}
\\
\|e^{t\Delta}f\|_{B_{p,1}^{s_1}} &\lesssim \left(1+t^{-\frac{1}{2}(s_1-s_0)}\right)\ln(e+t^{-1})\|f\|_{B_{p,\infty}^{s_0}} \label{eq:HeatNHBesovLog}
\end{align}
for all $t>0$. If $s_0<s_1$, $1\le p\le \infty$, then there holds
\begin{align}\label{eq:HeatHomBesovLog}
\|e^{t\Delta}f\|_{\dot B_{p,1}^{s_1}} &\lesssim t^{-\frac{1}{2}(s_1-s_0)}\|f\|_{\dot B_{p,\infty}^{s_0}}\ .
\end{align}
\end{lemma}
The proof of \eqref{eq:HeatHomBesov} and \eqref{eq:HeatNonHomBesov} is based on the $L^p$ estimate for each $\phi_j*f$ (mode of frequency) by using Young's inequality. The idea for \eqref{eq:HeatNHBesovLog} and \eqref{eq:HeatHomBesovLog} is to truncate the Besov norms by high frequency and low frequency terms and apply some interpolation inequalities in the Besov spaces.

\bigskip

Before showing results on equivalency and embedding, we provide a short introduction of Triebel-Lizorkin-type spaces. For the detailed definitions of Triebel-Lizorkin spaces and their variations, see \citet{Triebel2010} and \citet{Yuan2010}.
\begin{definition}
Let $s\in\bR$, $p\in(0,\infty)$ and $q\in(0,\infty]$. Let $\{\phi_j\}_{j\in\bZ}$ be Littlewood-Paley decomposition (homogeneous one) and $\cP(\bR^d)$ be the set of all polynomials. The \textit{Triebel-Lizorkin space} $\dot F^s_{p,q}(\bR^d)$ is the set of all $f\in \mathscr{S}'(\bR^d)/\cP(\bR^d)$ such that
\begin{align*}
\|f\|_{\dot F^s_{p,q}}:= \left\|\left(\sum_{j\in\bZ}\left(2^{js}|\phi_j*f|\right)^q\right)^{1/q}\right\|_{L^p(\bR^d)}<\infty\
\end{align*}
where the $\ell^q$-norm is replaced by the supremum on $j$ if $q=\infty$. For $p=\infty$, $\dot F^s_{p,q}(\bR^d)$ is defined to be the set of all $f\in \mathscr{S}'(\bR^d)/\cP(\bR^d)$ such that
\begin{align*}
\|f\|_{\dot F^s_{\infty,q}}:=\underset{Q\in\mathscr{D}}{\sup}\frac{1}{|Q|^{1/q}} \left\|\sum_{j=-\ln l(Q)}^\infty\left(2^{js}|\phi_j*f|\right)^q\right\|_{L^1(Q)}^{1/q}<\infty\
\end{align*}
where $\mathscr{D}$ denote the collection of all dyadic cubes and $l(Q)$ is the side length of $Q$.

Let $s,\tau\in\bR$, $p\in(0,\infty)$ and $q\in(0,\infty]$. Let $\{\phi_j\}_{j\in\bZ}$ be Littlewood-Paley decomposition (homogeneous one). The \textit{Triebel-Lizorkin-Morrey space} is defined to be the set of all $f\in \mathscr{S}'(\bR^d)/\cP(\bR^d)$ such that
\begin{align*}
\|f\|_{\dot F^{s,\tau}_{p,q}}:=\underset{Q\in\mathscr{D}}{\sup}\frac{1}{|Q|^\tau} \left\|\left(\sum_{j=-\ln l(Q)}^\infty\left(2^{js}|\phi_j*f|\right)^q\right)^{1/q}\right\|_{L^p(Q)}<\infty\
\end{align*}
where $\mathscr{D}$ and $l(Q)$ are as above. For $p=\infty$, modification of the norm should be made also.

The definitions of Triebel-Lizorkin spaces of non-homogeneous type are defined in a similar fashion (see \citet{Yuan2010} for details).
\end{definition}

\bigskip

The following two theorems connect $BMO$-type and Triebel-Lizorkin spaces. Both are proved by the equivalency $\mathfrak{h}^p=F^0_{p,2}$ (resp. $\cH^p=\dot F^0_{p,2}$) and the dualities $F^0_{\infty,2}\approx (F^0_{1,2})^*\approx (\mathfrak{h}^1)^*\approx bmo$ (resp. $\dot F^0_{\infty,2}\approx (\dot F^0_{1,2})^*\approx (\cH^1)^*\approx BMO$). For the details of the proofs, see \citet{Triebel2010}.
\begin{theorem}[\citet{Triebel2010}, Page 93, Theorem 2]
The following equality holds
\begin{align}\label{eq:bmoEqTLsp}
bmo(\bR^d)=F_{\infty,2}^0(\bR^d)
\end{align}
with norm equivalence.
\end{theorem}

\begin{theorem}[\citet{Triebel2010}, Theorem on Page~244, or \citet{Yuan2010}, Section~1.4.4]
The following equality holds
\begin{align}\label{eq:BMOEqTLsp}
BMO(\bR^d)=\dot F_{\infty,2}^0(\bR^d)
\end{align}
with equivalent quasi-norms.
\end{theorem}

\bigskip

Applying a result in \citet{Frazier1990} (representation of $\dot F_{p,q}^s$ in the sequence space indexed by the dyadic system $\mathscr{D}$), \citet{Yuan2010} proved that
\begin{theorem}[\citet{Yuan2010}, Proposition~2.4, or \citet{Frazier1990}, Corollary~5.7]
Let $s\in \bR$, $p\in(0,\infty)$ and $q\in (0,\infty]$. Then the followings hold with equivalent norms and equivalent quasi-norms respectively:
\begin{align}
F_{p,q}^{s,1/p}(\bR^d) &= F_{\infty,q}^s(\bR^d) \label{eq:NonHomTLspVarEq}
\\
\dot F_{p,q}^{s,1/p}(\bR^d) &=\dot F_{\infty,q}^s(\bR^d) \label{eq:HomTLspVarEq}
\end{align}
\end{theorem}

\bigskip

Now we are ready to prove the embedding results we quote in the proof of the main theorem.
\begin{lemma}\label{le:BMOEmbedBesov}
The following chains of continuous embeddings hold:
\begin{align}
&L^\infty \hookrightarrow bmo\hookrightarrow F_{\infty,\infty}^0=B^0_{\infty,\infty}
\\
&L^\infty \hookrightarrow bmo \hookrightarrow BMO\hookrightarrow \dot F_{\infty,\infty}^0=\dot B^0_{\infty,\infty}
\end{align}
\end{lemma}

\begin{proof}
It suffices to show $bmo\hookrightarrow F_{\infty,\infty}^0$ and $BMO\hookrightarrow \dot F_{\infty,\infty}^0$; the rest easily follows by the definitions of $BMO$ and $L^\infty$. In the spirit of \eqref{eq:bmoEqTLsp} and \eqref{eq:NonHomTLspVarEq}, to prove $bmo(\bR^d)\hookrightarrow F_{\infty,\infty}^0(\bR^d)$ is equivalent to showing
\begin{align*}
F_{\infty,2}^0(\bR^d)\hookrightarrow F_{p,\infty}^{0,1/p}(\bR^d)\qquad\textrm{for some }p\in(0,\infty)\ .
\end{align*}
Pick $p=2$. By the definitions of Triebel-Lizorkin spaces and the variations, $F_{\infty,2}^0(\bR^d)\hookrightarrow F_{2,\infty}^{0,1/2}(\bR^d)$ is a consequence of the monotonicity of $\ell^p$-norms and H\"older's inequalities. This proves $bmo(\bR^d)\hookrightarrow F_{\infty,\infty}^0(\bR^d)$. It follows similarly that $BMO\hookrightarrow \dot F_{\infty,\infty}^0$.
\end{proof}

\bigskip

\begin{proof}[Proof of Lemma~\ref{le:semigplocalbmo}]
The boundedness of $\|\nabla u(t)\|_{L^\infty}$, $\|\nabla^2 u(t)\|_{L^\infty}$ and $\|\partial_t u\|_{L^\infty}$ is a simple corollary of Lemma~\ref{th:semigpBMO} since $bmo\hookrightarrow BMO$. It remains to show $\|u(t)\|_{bmo}\lesssim \|u_0\|_{bmo}$. Similar to \citet{Bolkart2018}, we define the Hardy space $\mathfrak{h}^1(\bR^d)$ as
\begin{align*}
\left\{f\in L_{loc}^1(\bR^d)\ \bigg|\ \|f\|_{\mathfrak{h}^1}:= \left\|\underset{0<s<1}{\sup}|G_s*f|\right\|_{L^1}<\infty\right\}\ .
\end{align*}
Let $\phi\in \mathfrak{h}^1(\bR^d)$. Then for all $t$
\begin{align*}
\left|\langle u(t), \phi\rangle\right|\lesssim &\left|\langle u_0, G_t*\phi\rangle\right|
\\
\lesssim &\|u_0\|_{bmo}\|G_t*\phi\|_{\mathfrak{h}^1}
\\
\lesssim &\|u_0\|_{bmo}\left\|\underset{0<s<1}{\sup}|G_s*(G_t*\phi)|\right\|_{L^1}
\\
\lesssim &\|u_0\|_{bmo}\left\|G_t*\underset{0<s<1}{\sup}|G_s*\phi|\right\|_{L^1}
\\
\lesssim &\|u_0\|_{bmo}\left\|\phi\right\|_{\mathfrak{h}^1}\ .
\end{align*}
Then, $\|u(t)\|_{bmo}\lesssim \|u_0\|_{bmo}$ follows by $\mathfrak{h}^1-bmo$ duality.
\end{proof}

\bigskip

\begin{proof}[Proof of Lemma~\ref{le:RealAnalytic}]
First of all, by Duhamel's principle (for distributions) we know \eqref{eq:DuhPrin} solves the equation in $\mathscr{S}'(\bR^d\times (0,T))$, and for any $t_0\in(0,T)$ the function
\begin{align*}
\tilde{u}(t)=e^{(t-t_0)\Delta}u(t_0)+ \int_{t_0}^{t} \nabla e^{(t-s)\Delta} R(f_s)\ ds
\end{align*}
solves the equation in $\mathscr{S}'(\bR^d\times (t_0,T))$; moreover, $\tilde u$ and $u$ agree on $\bR^d\times (t_0,T)$. The analyticity of $\tilde{u}$ is due to the following estimates: By Lemma~\ref{le:SemGpBMO} and the assumption \eqref{eq:LeRANonHomAsmp}, for any $t>t_0$ we have
\begin{align*}
\|\nabla u(t)\|_\infty &\lesssim \|\nabla e^{(t-t_0)\Delta}u(t_0)\|_\infty + \left\|\nabla \int_{t_0}^{t} \nabla e^{(t-s)\Delta} R(f_s)\ ds \right\|_\infty
\\
&\lesssim (t-t_0)^{-1/2}\|u(t_0)\|_{BMO}\ + \int_{t_0}^{t} \|\nabla e^{(t-s)\Delta} \nabla R(f_s)\|_\infty ds
\\
&\lesssim (t-t_0)^{-1/2}\|u(t_0)\|_{BMO}\ + (t-t_0)^{-1/2}\sup_{s>t_0}\|\nabla R(f_s)\|_{BMO}
\\
&\lesssim (t-t_0)^{-1/2}\left(\|u_0\|_{BMO}\ + \sup_{s<t_0}a_0(s)\|f_s\|_{BMO} \int_0^{t_0}(t_0-s)^{-\frac{1}{2}}a_0(s)^{-1}ds\right)
\\
&\qquad\qquad\qquad\qquad\qquad\qquad + (t-t_0)^{-1/2} \sup_{s>t_0}\|\nabla f_s\|_\infty\ .
\end{align*}
Similarly, one can show, for higher order derivatives,
\begin{align*}
\|\nabla^k u(t)\|_\infty &\lesssim (t-t_0)^{-\frac{k}{2}}\left(\|u_0\|_{BMO}\ +\tilde{a}_1(t_0) \sup_{s<t_0}a_0(s)\|f_s\|_{BMO}\right) + (t-t_0)^{-\frac{1}{2}}\sup_{t_0<s<t}\|\nabla^k f_s\|_\infty
\end{align*}
for some weight function $\tilde{a}_1$. This proves that for any $t>t_0$, $\tilde{u}(t)$, hence $u(t)$, is analytic. Since $t_0$ is arbitrary, $u(t)$ is analytic for every $t\in (0,T)$.
\end{proof}

\begin{lemma}[Montel's]\label{le:Montel}
Let $p\in[1,\infty]$ and let $\mathscr{F}$ be a set of analytic functions $f$ in an open set $\Omega\subset \bC^d$ such that
\begin{align*}
\underset{f\in\mathscr{F}}{\sup}\ \|f\|_{L^p(\Omega)}<\infty\ .
\end{align*}
Then $\mathscr{F}$ is a normal family.
\end{lemma}

\bigskip

In the end, we give a proof of Lemma~\ref{le:HeatSmGpLfty}. To make the proof shorter, we borrow a result from \citet{Calderon1952}:
\begin{proposition}\label{prop:LlogLBddCZ}
Suppose $T$ is a Calder\'on-Zygmund operator with $K$ satisfying \eqref{eq:RadialCancel}. Let $f$ be a function such that
\begin{align}
\int_{\bR^d}|f(x)|\left(1+\ln^+|f(x)|\right)dx<\infty\ .
\end{align}
Then $Tf(x)$ is integrable over any set $S$ of finite measure and
\begin{align}
\int_S |Tf(x)|d\mu(x)\lesssim \int_{\bR^d}|f(x)|d\mu(x) + \int_{\bR^d}|f(x)|\ln^+\left(\mu(S)^{\frac{d+1}{d}}|f(x)|\right)d\mu(x) + \mu(S)^{-\frac{1}{d}}\
\end{align}
where $\mu$ denotes the Lebesgue measure in $\bR^d$.
\end{proposition}

\textit{Duality in Orlicz spaces:} It is well known that the Orlicz space $\phi(L)(\mu)$ is defined to be the function space with the norm
\begin{align*}
\|f\|_{\phi(L)}:=\inf\left\{s>0\ \big|\ \int_X\phi(s^{-1}|f|)d\mu\le 1\right\}
\end{align*}
where $\phi:\bR^+\to\bR^+$ is convex and increasing with $\phi(0)=0$. And the dual of $\phi(L)(\mu)$ is the Orlicz space $\psi(L)(\mu)$ (with the same norm) where $\psi$ is the Legendre-Fenchel transform of $\phi$ :
\begin{align*}
\psi(y):=\sup\left\{xy-\phi(x)\ |\ x\in\bR^+\right\}\ ,
\end{align*}
and vice versa. In particular, we set $\phi_*(x)=x\ln(e+x)$ and $\psi_*(x)=e^x-1$. Note that the Legendre-Fenchel transform of $\phi_*$ is not $\psi_*$ but comparable to $e^x$, however, if we consider the restriction of Lebesgue measure $\mu$ on a set $S$ of finite measure with $\mu(S)\lesssim 1$, then $\phi_*(L)(\mu|_{S})$ and $\psi_*(L)(\mu|_{S})$ are mutually dual spaces.

Such duality result, together with the ``quasi-boundedness'' of $\mathcal{CZO}$ in $L\log L$ as shown by Proposition~\ref{prop:LlogLBddCZ}, yields the following:
\begin{corollary}\label{cor:eLBddCZ}
For any set $S_1,S_2$ with finite Lebesgue measures such that $\mu(S_1)=c\mu(S_2)\lesssim1$ for some constant $c>0$, the $\cC$-$\cZ$ operator $T$ described in Lemma~\ref{le:HeatSmGpLfty} and Proposition~\ref{prop:LlogLBddCZ} is a bounded operator from $L^\infty(\mu|_{S_1})$ to $\psi_*(L)(\mu|_{S_2})$ where $\psi_*(x)=e^x-1$ and $\mu|_{S_i}$ denotes the restriction of Lebesgue measure on $S_i$. Moreover,
\begin{align*}
\|T\|_{L^\infty(\mu|_{S_1})\to \psi_*(L)(\mu|_{S_2})}\lesssim 1
\end{align*}
which is independent of $S_1$ and $S_2$.
\end{corollary}

\begin{proof}
Without loss of generality we assume $\mu(S_1)=1$ (In general we set $\|f\|_{\phi_*(L)(\mu|_{S_2})}=\mu(S_1)^{-\frac{1}{d}}$ for the argument below; also notice that the proof becomes trivial for the special case $\mu(S_1)=0$). First, recall that the dual of $\psi_*(L)(\mu|_{S_2})$ is $\phi_*(L)(\mu|_{S_2})$ where $\phi_*(x)=x\ln(e+x)$. By Proposition~\ref{prop:LlogLBddCZ}, for any $f\in \phi_*(L)(\mu|_{S_2})$ with $\|f\|_{\phi_*(L)(\mu|_{S_2})}=1$,
\begin{align*}
\int_{S_1} |T(f\1_{S_2})|d\mu(x)&\lesssim \int_{\bR^d}|f\1_{S_2}|d\mu(x) + \int_{\bR^d}|f\1_{S_2}|\ln^+\left(|f\1_{S_2}|\right)d\mu(x) + 1
\\
&\lesssim 2\|f\1_{S_2}\|_{\phi_*(L)(\mu)}\lesssim 2\|f\|_{\phi_*(L)(\mu|_{S_2})} \lesssim 1\ .
\end{align*}
Then, by the above estimate and the self-adjointness of $T$, it follows that for any $g\in L^\infty(\mu|_{S_1})$ and $f\in \phi_*(L)(\mu|_{S_2})$ with $\|f\|_{\phi_*(L)(\mu|_{S_2})}=1$,
\begin{align*}
\left|\int_{S_2} f(x) T(g\1_{S_1})(x)d\mu(x)\right|&=\left|\int_{\bR^d} T(f\1_{S_2})(y) g\1_{S_1}(y)d\mu(y)\right|
\\
&\lesssim \|g\1_{S_1}\|_{L^\infty}\int_{S_1} |T(f\1_{S_2})|d\mu(y)\lesssim \|g\1_{S_1}\|_{L^\infty}\ .
\end{align*}
Then, by the duality $\psi_*(L)(\mu|_{S_2})\approx\left(\phi_*(L)(\mu|_{S_2})\right)^*$, it follows that
\begin{align*}
\|T(g\1_{S_1})\|_{\psi_*(L)(\mu|_{S_2})}\lesssim \|g\1_{S_1}\|_{L^\infty}\ ,
\end{align*}
in other words, $T: L^\infty(\mu|_{S_1})\to \psi_*(L)(\mu|_{S_2})$ is bounded.
\end{proof}

\bigskip

\begin{proof}[Proof of Lemma~\ref{le:HeatSmGpLfty}]
First we prove for $k=1$. Let $f_x(y)$ denote the translation $f(x-y)$, then for an open ball $B$ centered at 0 with radius $r_B$, we have the decomposition
\begin{align*}
\left|g_t*|Tf|\right|&=\left|\int_{\bR^d}g_t(y)|Tf_x(y)|dy\right|
\\
&\lesssim \int_{B^c}|g_t||Tf_x|dy\ + \int_{B}|g_t|\left|T\left(f_x\1_{(3B)^c}\right)\right|dy\ + \int_{B}|g_t|\left|T\left(f_x\1_{3B}\right)\right|dy=:H+I+J\
\end{align*}
where $\kappa B$ is $\kappa$-multiple dilation of $B$ from the center.
If $p>1$, since $B$ is centered at 0, i.e. $c_B=0$, by the assumption on $g$ and $\Phi$, H\"older's inequality and $L^p$-boundedness of $\mathcal{CZO}$,
\begin{align*}
H &\lesssim \int_{B^c}\Phi_t(y)|Tf_x(y)|dy
\\
&\lesssim \|\Phi_t\|_{L^{p'}(B^c)}\|Tf_x\|_{L^p(B^c)}
\\
&\lesssim |\Phi_t(r_B)|^{\frac{p'-1}{p'}}\left(\|\Phi_t\|_{L^1}\right)^{1/p'}\|f_x\|_{L^p}
\\
&\lesssim |\Phi(r_B/t)|^{\frac{p'-1}{p'}}\left(\|\Phi\|_{L^1}\right)^{1/p'}\|f\|_{L^p}\ .
\end{align*}
If $p=1$, pick a $q$ such that $p<q<\infty$, then by the above argument with $L^p$ interpolations
\begin{align*}
H &\lesssim \int_{B^c}\Phi_t(y)|Tf_x(y)|dy
\\
&\lesssim |\Phi(r_B/t)|^{\frac{q'-1}{q'}}\left(\|\Phi\|_{L^1}\right)^{1/q'}\|f\|_{L^q}
\\
&\lesssim |\Phi(r_B/t)|^{\frac{q'-1}{q'}}\left(\|\Phi\|_{L^1}\right)^{1/q'}\|f\|_{L^\infty}^{(q-p)/q}\|f\|_{L^p}^{p/q}\ .
\end{align*}
By the ``size'' condition of the $\cC$-$\cZ$ kernel $K$ and H\"older's inequality
\begin{align*}
\left|T\left(f_x\1_{(3B)^c}\right)(c_B)\right|\lesssim r_B^{-(p'-1)d/p'} \|f_x\|_{L^p}\lesssim r_B^{-(p'-1)d/p'} \|f\|_{L^p}\ .
\end{align*}
For any $y\in B$, by the ``smoothness'' condition of $K$,
\begin{align*}
\left|T\left(f_x\1_{(3B)^c}\right)(y)-T\left(f_x\1_{(3B)^c}\right)(c_B)\right|&\lesssim\int_{(3B)^c} \left|K(y,z)-K(c_B,z)\right||f_x(z)|dz
\\
&\lesssim\int_{(3B)^c} \frac{|y-c_B|^\delta}{|z-c_B|^{d+\delta}}|f_x(z)|dz\lesssim_\delta \|f\|_{L^\infty}  \ .
\end{align*}
Therefore $\displaystyle I\lesssim_{cz} \|f\|_{L^\infty} + r_B^{-(p'-1)d/p'} \|f\|_{L^p}$. By the duality of the Orlicz spaces,
\begin{align*}
J\lesssim \|g_t\|_{\phi_*(L)(\mu|_B)}\|T\left(f_x\1_{3B}\right)\|_{\psi_*(L)(\mu|_B)}
\end{align*}
where $\phi_*$ and $\psi_*$ are given in Corollary~\ref{cor:eLBddCZ}. Now, by the corollary (being applied with $S_1=3B$ and $S_2=B$ as well as $r_B\approx 1$),
\begin{align*}
J &\lesssim \|g_t\|_{\phi_*(L)}\|f_x\1_{3B}\|_{L^\infty}
\\
&\lesssim \|f\|_{L^\infty} \int_{\bR^d}|g_t(x)|\ln(e+|g_t(x)|)dx
\\
&\lesssim_d \ln(e+t^{-1})\ \|f\|_{L^\infty} \int_{\bR^d}\Phi(x)\ln(e+\Phi(x))dx\ .
\end{align*}
To sum up, we have shown that, for some constants $\alpha,\beta,\gamma>0$,
\begin{align*}
\left\|g_t*|Tf|\right\|_{L^\infty}&\lesssim \left(|\Phi(r_B/t)|^{1-\beta}\|\Phi\|_{L^1}^\beta + 1+ r_B^{-\gamma} + \|\Phi\|_{\phi_*(L)}\ln(e+t^{-1})\right)
\\
&\qquad\qquad\qquad\qquad \times\left(\|f\|_{L^\infty}+\|f\|_{L^p}+\|f\|_{L^\infty}^\alpha\|f\|_{L^p}^{1-\alpha}\right)\ .
\end{align*}
If $t\le r_B\approx 1$, then, by the decreasing property of $\Phi$,
\begin{align*}
\left\|g_t*|Tf|\right\|_{L^\infty}&\lesssim \left(1 + |\Phi(1)|^{1-\beta}\|\Phi\|_{L^1}^\beta + \|\Phi\|_{\phi_*(L)}\ln(e+t^{-1})\right)
\\
&\qquad\qquad\qquad\qquad \times\left(\|f\|_{L^\infty}+\|f\|_{L^p}+\|f\|_{L^\infty}^\alpha\|f\|_{L^p}^{1-\alpha}\right)\
\end{align*}
which proves the lemma for $k=1$. The proof for $k\neq1$ is similar: We still do the decomposition in the first step, that is
\begin{align*}
\left|g_t*|Tf|^k\right|&\lesssim_k \int_{B^c}|g_t||Tf_x|^k dy\ + \int_{B}|g_t|\left|T\left(f_x\1_{(3B)^c}\right)\right|^k dy\ + \int_{B}|g_t|\left|T\left(f_x\1_{3B}\right)\right|^k dy
\\
&=:H+I+J\ .
\end{align*}
The estimations for $H$ and $I$ are completely the same as those for $k=1$. In order to estimate $J$, we need to modify the Orlicz spaces we used in the proof, i.e. we take $\psi_k(x)=e^{x^{1/k}}-1$ instead of $\psi_*(x)=e^x-1$. Then the corresponding $\phi_k$ (Legendre-Fenchel transform of $\psi_k$) is still some function growing slightly faster than linear up to a logarithmic factor which is approximately $(\ln(e+x))^k$ as $x$ gets larger. Thus, Corollary~\ref{cor:eLBddCZ} and the definition of $\psi_k$ yields similar estimates:
\begin{align*}
J &\lesssim \|g_t\|_{\phi_k(L)(\mu|_B)}\left\||T\left(f_x\1_{3B}\right)|^k\right\|_{\psi_k(L)(\mu|_B)}
\\
&\lesssim \|\Phi_t\|_{\phi_k(L)(\mu|_B)}\left(\left\|T\left(f_x\1_{3B}\right)\right\|_{\psi_*(L)(\mu|_B)}\right)^k
\\
&\lesssim \int_{\bR^d}\Phi_t(x)(\ln(e+\Phi_t(x)))^kdx \ \|f_x\1_{3B}\|_{L^\infty}^k
\\
&\lesssim_d (\ln(e+t^{-1}))^k \int_{\bR^d}\Phi(x)(\ln(e+\Phi(x)))^kdx \ \|f\|_{L^\infty}^k\ .
\end{align*}
\end{proof}

\end{appendix}

\section*{Acknowledgments}
This question was developed during my graduate study at the University of Virginia Math Department with the instructions by Professor Zoran Gruji\'c, my program advisor. I would give thanks to Professor Gruji\'c, for offering a topic course in the Navier-Stokes equations, from the formulation of weak solutions to the existence of strong solutions in $L^\infty(0,T; L^p)$ with $p>3$ and $T$ depending on the $L^p$-norm of the initial data, i.e. the proof of existence in $L^\infty(0,T; L^p)$ and how existence in $L^\infty(0,T; L^p)$ results in smoothness up to $T$ and the further discussions with him about the model, more specifically, the sharpest result on the regularity criterion in the $L^p$-functional framework seemed to be $u \in L^p (0,T; L^q)$ with $\displaystyle\frac{3}{q} + \frac{2}{p} = 1$ or $u \in L^\infty (0,T; L^3)$ i.e. the criteria with scale invariant norms without hypothesis, or $u \in L^\infty (0,T; L_w^3)$ and other scale invariant norms with restrictions while the best result for the \textit{a priori} bound seemed to be $u \in L^p (0,T; L^q)$ with $\displaystyle\frac{3}{q} + \frac{2}{p} = \frac{3}{2}$ or $\omega \in L^\infty (0,T; L^1)$, followed by the notion of `scaling gap'.
Which approaches seemed more realistic? Any direct method to prove smoothness within arbitrarily large time span? To improve regularity from the weak solution or to extend the strong solution initiated from large initial values or both? How logarithmic sub-criticality, thus regularity, is achievable with coherence of vorticity direction in the vortex stretching scenario? And the mathematical setup for the sparsity of intense velocity or vorticity and how sparseness of the super-level sets results in intermittency of the magnitude, as demonstrated in \cite{Bradshaw2019} and \cite{Grujic2013}.
And his reading suggestions \cite{Foias2001}, \cite{LemarieRieusset2002}, \cite{Lions1996}, \cite{Caffarelli1982}, \cite{Ransford1995}, \cite{Iskauriaza2003}, \cite{Dascaliuc2012}, \cite{Cianchi2003}, \cite{Constantin1993}, \cite{BeiraodaVeiga2002}, \cite{Bradshaw2014}, \cite{Bradshaw2015}, \cite{Farhat2017}, \cite{Bradshaw2019}, \cite{Grujic2013}.  And \textit{much freedom for self-explorations he approved of in my graduate study.} I express respect and gratitude to Zoran Grujic for the patience in discussing the complexification method for proving analyticity and for fixing some language expressions at the completion of the first draft.  The acknowledgment in the first draft was non-exact.  I am grateful that the University of Virginia and UVA Math Department grant me half-year extension of the program.


\def\cprime{$'$}

\end{document}